\documentclass[a4paper,11pt]{article}
\usepackage{bbm}
\usepackage{}
\usepackage{mathrsfs}
\usepackage{amsfonts}
\usepackage{amsfonts}
\usepackage{amssymb}
\usepackage{mathrsfs}
\usepackage{indentfirst,latexsym,bm,amsmath,amssymb,amsthm}
\usepackage{xcolor}
\usepackage[%CJKbookmarks=true,%
                   bookmarksnumbered=true,
                   bookmarksopen=true, %Collapse all bookmarks at pdf start view
                   pdfauthor=kmc,
                   pdfcreator={LaTex with hyperref package + WinEdt},
                   pdftitle=trudge,
                   colorlinks,%
                   citecolor=blue,
                   linkcolor=blue,%
                   urlcolor=blue,
                   hyperindex,%
                   plainpages=false,%
                   pdfstartview=FitH,
                   linktocpage=true,
                   dvipdfm%
                   ]{hyperref}
\usepackage{indentfirst,latexsym,bm,amsmath,amssymb,amsthm}
\usepackage[dvips]{graphicx}

\makeatletter\@addtoreset{equation}{section} \makeatother
\newtheorem{thm}{Theorem}[section]
\newtheorem{Lemma}{Lemma}[section]

\newtheorem{rem}{Remark}[section]

\makeatletter 

\allowdisplaybreaks
\makeatother

\title{Global stability of traveling waves for $(1+1)$-dimensional systems of quasilinear wave equations}

\author{Louis Dongbing Cha (Dongbing Zha)\thanks{Department of Mathematics and Institute for Nonlinear Sciences, Donghua University, Shanghai 201620, PR China. {E-mail address: ZhaDongbing@163.com}, LouisDongbingCha@gmail.com}~~~~
Arick Shao\thanks{School of Mathematical Sciences, Queen Mary University of London, Mile End Road, London E1 4NS, United Kingdom. {E-mail address: a.shao@qmul.ac.uk}} }

\begin{document}

\maketitle
\begin{abstract}
A key feature of $(1+1)$-dimensional nonlinear wave equations is that they admit left or right traveling waves, under appropriate algebraic conditions on the nonlinearities.
In this paper, we prove global stability of such traveling wave solutions for $(1+1)$-dimensional systems of nonlinear wave equations, given a certain asymptotic null condition and sufficient decay for the traveling wave.
We first consider semilinear systems as a simpler model problem; we then proceed to treat more general quasilinear systems.
\end{abstract}

\pagestyle{plain} \pagenumbering{arabic}

\section{Introduction}

Let $( t, x ) \in \mathbb{R} \times \mathbb{R}$ denote the usual Cartesian coordinates, and define also
\begin{equation}
\label{eq.uv} \xi=\frac{t+x}{2} \text{,} \qquad \eta=\frac{t-x}{2} \text{,}
\end{equation}
as well as the corresponding vector fields
\begin{equation}
\label{eq.uv_vf} \partial_{\xi} = \partial_{t}+ \partial_{x} \text{,} \qquad \partial_{\eta}= \partial_{t}-\partial_{x} \text{.}
\end{equation}
Consider the following $(1+1)$-dimensional system of quasilinear wave equations,
\footnote{Note that the left-hand side $v_{ \xi \eta } = v_{ t t } - v_{ x x }$ of \eqref{quasiwave} represents the standard wave operator, while the right-hand side of \eqref{quasiwave} represents nonlinear perturbations.}
\begin{align}\label{quasiwave}
v_{\xi\eta}&=A_1(v_{\xi},v_{\eta})v_{\xi\eta} + A_2(v_{\xi},v_{\eta})v_{\eta\eta}+A_3(v_{\xi},v_{\eta})v_{\xi\xi}+F(v_{\xi},v_{\eta}) \text{,}
\end{align}
where $v = v(t,x): \mathbb{R}^{1+1} \longrightarrow \mathbb{R}^{n}$ is the unknown; where $A_i: \mathbb{R}^{n} \times \mathbb{R}^{n} \rightarrow \mathbb{R}^{n\times n}$ ($i = 1, 2, 3$) are given smooth and matrix-valued functions; and where $F: \mathbb{R}^{n} \times \mathbb{R}^{n} \rightarrow \mathbb{R}^{n}$ is a given smooth and vector-valued function.
We also assume that $A_1, A_2, A_3$ are symmetric.

In addition, we assume the following for the coefficients $A_i$ and $F$:
\begin{align}
\label{order1} A_1(\rho, \theta) &= \mathscr{O}(|\rho|+|\theta|) \text{,} \\
\label{order22} A_2(\rho, \theta) &= \mathscr{O}(|\rho|) \text{,} \\
\label{order33} A_3(\rho, \theta) &= \mathscr{O}(|\theta|) \text{,} \\
\label{order4} F(\rho,\theta) &= \mathscr{O}(|\rho||\theta|) \text{.}
\end{align}
The assumptions \eqref{order1}--\eqref{order4} correspond to the null condition in the small data setting, within which global existence of classical solutions for \eqref{quasiwave} is proved in \cite{MR4098041}. We should point out that \cite{MR4098041} is inspired by the corresponding result in the semilinear case \cite{MR3783412}, which strengthens a former result in \cite{MR3121700}. Presently, \eqref{order1}--\eqref{order4} remain closely related to null conditions; we discuss this point further later.

%\begin{rem}
%Since $A_1$ is smooth, \eqref{order1} implies $A_1$ can be written in the form
%\[
%\label{order1_alt} A_1 ( \rho, \theta ) = a_{ 1, \rho } ( \rho, \theta ) \cdot \rho + a_{ 1, \theta } ( \rho, \theta ) \cdot \theta \text{,}
%\]
%with matrix-valued functions $a_{ 1, \rho }$ and $a_{ 1, \theta }$, for which all of their derivatives are locally bounded.
%Similar observations also hold for $A_2$, $A_3$, $F$.
%\end{rem}

Under \eqref{order1}--\eqref{order4}, one can verify that the left and right traveling waves $v=f(\xi)$ and $v=g(\eta)$ (with $f, g$ taking values in $\mathbb{R}^n$) are solutions to \eqref{quasiwave}.
Thus, a natural problem is to ask whether they are globally stable under some suitable assumptions.

\subsection{Semilinear Systems}

To simplify the discussion, let us first consider a semilinear system (i.e., with $A_i = 0$).
Consider the following $(1+1)$-dimensional system of semilinear wave equations,
\begin{align}\label{quasiwave888}
v_{\xi\eta}&=F(v_{\xi},v_{\eta}) \text{,}
\end{align}
where the condition \eqref{order4} is satisfied.
We can, without loss of generality, restrict our attention to stability for the left traveling wave $f ( \xi )$.
Let us write
\begin{equation}
\label{eq.u} v=u+f(\xi) \text{.}
\end{equation}
Then, $u$ satisfies the modified system
\begin{align}\label{vsystem999}
u_{\xi\eta}&=F(f'(\xi)+u_{\xi},u_{\eta}) \text{.}
\end{align}

In addition, we fix $0 < \delta < 1$, and we assume
\begin{align}\label{fgty7888}
M_0 = \sup_{ x \in \mathbb{R} } \left[ \langle x\rangle^{\frac{3}{2}(1+\delta)}(|f'(x)|+|f''(x)|) \right] < +\infty \text{,}
\end{align}
where $\langle \cdot\rangle=(1+|\cdot|^2)^{1/2}$.
Consider the Cauchy problem of \eqref{vsystem999}, with smooth initial data
\begin{align}\label{ddffk789yh}
u |_{ t=0 } = u_0 \text{,} \qquad u_t |_{ t = 0 } = u_1 \text{.}
\end{align}
As our preliminary result, we establish the following:

\begin{thm}\label{thn88786hhh9}
Consider the Cauchy problem \eqref{vsystem999}, \eqref{ddffk789yh}.
Assume the condition \eqref{order4} is satisfied, and assume the decay property \eqref{fgty7888} holds for some $0 < \delta < 1$.
Then there exists a constant $\varepsilon_0 > 0$, depending on $\delta$ and $M_0$, such that for any $0<\varepsilon\leq \varepsilon_0$, if
\begin{align}
\label{small_semilinear} \sum_{l=0}^{1} \left(\|\langle x\rangle^{1+\delta}\partial_x^{l}\partial_xu_0\|_{L_{x}^{2}(\mathbb{R})}+\|\langle x\rangle^{1+\delta}\partial_x^{l}u_1\|_{L_{x}^{2}(\mathbb{R})} \right)\leq \varepsilon \text{,}
\end{align}
then the problem \eqref{vsystem999}, \eqref{ddffk789yh} admits a unique global classical solution $u$.
In other words, the traveling wave solution $f(\xi)$ to \eqref{quasiwave888} is globally nonlinearly stable.
 \end{thm}

\subsection{Quasilinear Systems}

Next, we turn to the full quasilinear setting, that is, we consider the system \eqref{quasiwave}, along with the conditions \eqref{order1}--\eqref{order4}.
Again, we set
\begin{align}
v=u+f(\xi) \text{,}
\end{align}
so that $u$ satisfies
\begin{align}\label{vsystem99999}
u_{\xi\eta} &= A_1(f'(\xi)+u_{\xi},u_{\eta})u_{\xi\eta}+A_2(f'(\xi)+u_{\xi},u_{\eta})u_{\eta\eta}+A_3(f'(\xi)+u_{\xi},u_{\eta})u_{\xi\xi} \nonumber\\
 &\qquad + A_3(f'(\xi)+u_{\xi},u_{\eta})f^{''}(\xi)+F(f'(\xi)+u_{\xi},u_{\eta}) \text{.}
\end{align}
In addition, fixing $0 < \delta < 1$, we assume
\begin{align}\label{fgty788899}
M_1 = \sup_{x\in \mathbb{R}} \left[ \langle x\rangle^{3(1+\delta)}(|f'(x)|+|f''(x)|+|f^{(3)}(x)|+|f^{(4)}(x)|) \right] < +\infty \text{.}
\end{align}

Note that, in view of \eqref{order1}--\eqref{order4}, the linearized system of \eqref{vsystem99999} about the zero solution (i.e., the linearized system of \eqref{quasiwave} with respect to $f(\xi)$) is
\begin{align}\label{vsystem9999900gg0}
u_{\xi\eta} &= A_1(f'(\xi),0) u_{\xi\eta} + A_2(f'(\xi),0) u_{\eta\eta} +{\text{lower-order terms.}}
\end{align}
To ensure the hyperbolicity of \eqref{vsystem9999900gg0}, we must assume there exists $\lambda>0$ such that
\begin{align}
\label{hyper1} y^{T} [ I- A_1(f'(\xi),0)-A_2(f'(\xi),0) ] y &\geq \lambda |y|^2 \text{,} \\
\label{hyper3} y^{T} [ I-A_1(f'(\xi),0) ] y &\geq \lambda|y|^2 \text{,}
\end{align}
for any $\xi \in \mathbb{R}$ and $y \in \mathbb{R}^n$.

Consider now the Cauchy problem for \eqref{vsystem99999}, with smooth initial data
 \begin{align}\label{ddffk789yh99}
u |_{ t=0 } = u_0 \text{,} \qquad u_t |_{ t = 0 } = u_1 \text{.}
 \end{align}
The main result of this paper is the following:

\begin{thm}\label{thn88786hhh9900}
Consider the Cauchy problem \eqref{vsystem99999}, \eqref{ddffk789yh99}.
Assume \eqref{order1}--\eqref{order4}, the decay property \eqref{fgty7888} (for some $0 < \delta < 1$), and the hyperbolicity conditions \eqref{hyper1}, \eqref{hyper3}.
Then, there is a constant $\varepsilon_0 > 0$---depending on $\delta$, $M_0$, and $\lambda$---such that for any $0<\varepsilon\leq \varepsilon_0$, if
\begin{align}
\label{small_quasilinear} \sum_{l=0}^{2} \left( \|\langle x\rangle^{1+\delta}\partial_x^{l}\partial_xu_0\|_{L_{x}^{2}(\mathbb{R})}+\|\langle x\rangle^{1+\delta}\partial_x^{l}u_1\|_{L_{x}^{2}(\mathbb{R})} \right) \leq \varepsilon \text{,}
\end{align}
then \eqref{vsystem99999}, \eqref{ddffk789yh99} admits a unique global classical solution $u$.
In other words, the traveling wave solution $f(\xi)$ to \eqref{quasiwave} is globally nonlinearly stable.
\end{thm}

In fact, Theorem \ref{thn88786hhh9900} can be extended to a wider class of systems:

\begin{rem}
Theorem {\rm{\ref{thn88786hhh9900}}} still holds when $A_1, A_2, A_3, F$ also depend on $v$ itself.
We omit the details in this paper to make the exposition simpler, but the methods are analogous to those in {\rm{\cite{MR4098041}}}, which proved global stability in this more general setting for small initial data.
\end{rem}

\begin{rem} \label{general}
Furthermore, by inspecting its proof, one can observe that Theorem {\rm{\ref{thn88786hhh9900}}} still holds when when the conditions \eqref{order22} and \eqref{order4} are relaxed to
\begin{align}
A_2&=\mathscr{O}(|\rho|+|h(\xi)||\theta|) \text{,} \\
F&=\mathscr{O}((|\rho|+|h(\xi)||\theta|)|\theta| \text{,}
\end{align}
where $h$ admits a suitable decay property, i.e.,
\begin{align}
\sup_{x \in \mathbb{R}} \left[ \langle x\rangle^{3(1+\delta)}(|h(x)|+|h'(x)|+|h''(x)|) \right] < + \infty \text{.}
\end{align}
\end{rem}

\subsection{Previous Results}

The question of global existence for nonlinear wave equations with small initial data has been under active investigation for the past four decades.
Generally speaking, in higher dimensions, the mechanism for treating this problem is based on the decay in time of linear waves.
In particular, it is well-known that the decay rate for $(d+1)$-dimensional linear waves is $(1+t)^{-\frac{d-1}{2}}$.
When $d\geq 4$, small data global existence results hold for generic quadratic or higher-order nonlinearities \cite{MR544044,Klainerman86}, since this decay rate is integrable in time.

When $d\leq 3$, however, classical solutions may blow up in finite time \cite{John1}.
Thus, Klainerman \cite{Klainerman82} introduced the so-called null condition to rule out the worst-behaving nonlinear terms that drive the formation of singularities.
Under this condition, global existence for small data was established in pioneering works \cite{Klainerman86, MR820070} in the $(3+1)$-dimensional case, and in \cite{Alinhac01} in the $(2+1)$-dimensional case (see also \cite{MR3729247,MR3912654} for more general settings and related results \cite{Cheng,Dong,Hou,Lihao,Peng,ZhaZha}).
In contrast, linear waves in one spatial dimension do not decay in time.
Nonetheless, for $(1+1)$-dimensional semilinear wave equations, Luli et al.~\cite{MR3783412} proved that small data still leads to global solutions if the null condition is satisfied. See also a former result in \cite{MR3121700}.
The mechanism for global existence here is the interaction of waves with different velocities, which will lead to the decay of nonlinear terms. The authors capture this mechanism in the $L^2$-framework, that is, they develop weighted energy estimates with positive weights, which are robust and can have many other applications. The quasilinear case was treated in \cite{MR4098041}, based on the weighted energy estimates with positive weights in \cite{MR3783412}, some space-time weighted energy estimates, and new observations on the null structure in the quasilinear part.

For $(1+1)$-dimensional first-order quasilinear hyperbolic systems which are strictly hyperbolic or admit characteristics with constant multiplicity, a powerful method is the characteristic approach (see {\rm{\cite{MR2309570,MR1284811,MR2045426,MR3730736,MR2544109}}}).
However, for $(1+1)$-dimensional second\emph{}-order quasilinear systems \eqref{quasiwave}, the characteristic approach does not generally work when $n > 1$ (non-scalar case).
In particular, if one reduces \eqref{quasiwave} to a first-order system, then it may not be strictly hyperbolic and may admit characteristics without constant multiplicity, which may fail to be differentiable with respect to their variables.
As a result, in this paper, we will develop energy-based methods based on \cite{MR3783412} and \cite{MR4098041}.

For previous results on the global stability of traveling waves for $(1+1)$-dimensional quasilinear hyperbolic systems, we refer the reader to \cite{MR3057301} for first-order quasilinear strictly hyperbolic systems with linearly degenerate characteristic fields (see also \cite{MR3223827}), and \cite{Wong22} for some scalar quasilinear wave equations with cubic nonlinearity satisfying double null conditions.
We note, however, that all these results are based on characteristic approaches.
We also note that some results in higher dimensions are shown in \cite{Wong33}, \cite{Anderson} and \cite{zhou}.

\subsection{Relation to Null Conditions}

Recall that in small data settings, \eqref{order1}--\eqref{order4} corresponds to the null condition for \eqref{quasiwave}.
Here, one views $v_\ast \equiv 0$ as a reference solution for \eqref{quasiwave}, for which $\partial_\xi$ and $\partial_\eta$ are both characteristic directions for every component of the system.
In particular, \eqref{order1}--\eqref{order4} rule out unfavorable nonlinear combinations of these characteristic derivatives.

In Theorem \ref{thn88786hhh9900}, however, the reference solution is now the traveling wave, $v_f = f ( \xi )$, hence the appropriate null conditions deal instead with characteristics associated with $v_f$.
Note that the ``leftward-pointing" $\partial_\eta$ (along level sets of $\xi$) remains a common characteristic direction for all components of the system \eqref{quasiwave}.
On the other hand, when $f ( \xi )$ is not small, $\partial_\xi$ can deviate significantly from the ``rightward" characteristic directions of $v_f$.
Moreover, various components of $v_f$ can have very different rightward characteristic directions.

As a result, \eqref{order1}--\eqref{order4} cannot quite be interpreted as a null condition, although it still has a closely connected meaning.
In particular, as $| \xi | \rightarrow \infty$, so that $f ( \xi )$ decays to $0$ by \eqref{fgty788899}, the rightward characteristic directions of all components of our system must asymptote to $\partial_\xi$.
Thus, one can view \eqref{order1}--\eqref{order4} as a genuine lefward-null condition for $\partial_\eta$-direction, and as an asymptotic null condition for the $\partial_\eta$-direction.
In particular, our results show that this partial null condition is sufficient for global stability.

\begin{rem}
Because of this asymptotic nature of \eqref{order1}--\eqref{order4}, we can also think of this as a special case of the weak null condition; see {\rm{\cite{LR03,MR2680391,Keir18}}} for details.
\end{rem}

\subsection{Main Features}

It is also instructive to compare our setting with that of \cite{Wong22}, which also studied the global stability of traveling waves for quasilinear wave equations.
First, since \cite{Wong22} dealt with scalar equations, they were able to effectively semilinearize, and hence simplify, the problem by changing to characteristic coordinates.
This cannot be done in the proof Theorem \ref{thn88786hhh9900}, as different components of our system can have very different characteristic directions.

We also mention that the nonlinear terms in \eqref{quasiwave} are weaker than those studied in \cite{Wong22}. Specifically speaking, even in the small data setting, the system studied in \cite{Wong22} is cubic and satisfies double null conditions, while the system \eqref{quasiwave} is quadratic and satisfies weaker algebraic properties.
Therefore, when considering some perturbation of traveling waves, one can expect worse behavior from these quadratic terms.
One consequence is that solutions to \eqref{vsystem99999} can exhibit a linear effect, in which the perturbation $u$ is exponentially amplified as it propagates transversely through the traveling wave $f ( \xi )$.
However, since this amplification is localized in $\xi$, due to the decay of $f ( \xi )$ from \eqref{fgty788899}, this phenomenon cannot lead to blowup as long as the initial perturbation is sufficiently small.

This same mechanism also applies to the nonlinear analysis.
While there are nonlinear terms that are far from satisfying a null condition, these dangerous terms must again be localized in $\xi$ and hence cannot lead to blowup.
(This is also the idea behind the generalization described in Remark \ref{general}.
Although one allows for additional nonlinearities that seem potentially dangerous, these terms are always localized in $\xi$.)

Finally, as mentioned before, Theorem \ref{thn88786hhh9900} is proved by developing energy approaches inspired by \cite{MR3783412} and \cite{MR4098041}.
The main innovations in our approach can be explained as follows.
First, a key difference from \cite{MR3783412} and \cite{MR4098041} is that the traveling wave yields a hierarchical structure for \eqref{vsystem99999}, in which derivatives of $u$ must be treated in a specific order.
In general, $\eta$-derivatives of $u$ (approximately along the traveling wave) must be controlled first, while the remaining $\xi$-derivatives can only be bounded after the $\eta$-derivatives are controlled.
Second, in the treatment of $\eta$-derivatives of $u$, in the top-order energy estimate, because we must first control linear terms, the approach of control for quasilinear terms is new and different from the corresponding one in the small data setting \cite{MR4098041}.

\subsection{Outline}

The outline of this paper is as follows.
In Section \ref{sec2}, we introduce various notations that will be used in the sequel, as well as some tools required for the proofs of Theorems \ref{thn88786hhh9} and \ref{thn88786hhh9900}.
Sections \ref{sec3} and \ref{sec4} are then devoted to the proofs of Theorems \ref{thn88786hhh9} and \ref{thn88786hhh9900}, respectively.

\section{Preliminaries}\label{sec2}

The purpose of this section to establish some notations and basic estimates that will be useful throughout the proofs of Theorems \ref{thn88786hhh9} and \ref{thn88786hhh9900}.
Throughout, we will use $C$ to denote universal positive constants whose values can change between lines.

\subsection{Notations} \label{secP}

Following \cite{MR3783412}, we first define several spacetime regions in $\mathbb{R}^{1+1}$.
Fix $t_0>0$ and $x_0\in \mathbb{R}$, and let $\xi_0=\frac{t_0+x_0}{2}$.
We then define the leftward segment $C^{-}_{t_0,\xi_0}$ as
\begin{align}
C^{-}_{t_0,\xi_0}&=\{(t,x): 0\leq t\leq t_0, x=2\xi_0-t\} \text{.}
\end{align}
In addition, we define the spatial segments
\begin{align}
\Sigma^{-}_{t_0,\xi_0}&=\{(t,x): t=t_0, x\leq 2\xi_0-t_0\} \text{,}
\end{align}
as well as the spacetime regions
\begin{align}
S^{-}_{t_0,\xi_0}&=\{(t,x): 0\leq t\leq t_0, x\leq 2\xi_0-t\} = \bigcup_{\xi\leq \xi_0}C^{-}_{t_0,\xi} \text{.}
\end{align}
Moreover, we use $\Sigma_{t_0}$ to denote the full time slice,
\begin{align}
\Sigma_{t_0} &= \{(t,x): t=t_0, x\in \mathbb{R}\} \text{,}
\end{align}
and we write $S_{t_0}$ to denote the spacetime strip
\begin{align}
S_{t_0}&=\{(t,x): 0\leq t\leq t_0, x\in \mathbb{R}\} \text{.}
\end{align}

By convention, we interpret integrals over $C^-_{ t_0, \xi_0 }$ as line integrals in the Euclidean sense,
\begin{align}
\int_{ C^-_{ t_0, \xi_0 } } w = \sqrt{2} \int_{ -\xi_0 }^{- \xi_0 + t_0 } w ( \xi_0 + \eta, \xi_0 - \eta ) \, d \eta \text{,}
\end{align}
where $w$ is expressed as a function of $( t, x )$.
Using Fubini's theorem, one can verify that
\begin{align}
\label{key1} \int_{ S^-_{t_0,\xi_0} } w (t,x) dxdt &= \sqrt{2} \int_{-\infty}^{\xi_0} \left( \int_{C^{-}_{t_0,\xi}} w \right) d\xi \text{,} \\
\label{key2} \int_{ S_{t_0} } w (t,x) dxdt &= \sqrt{2} \int_{\mathbb{R}} \left( \int_{ C^-_{t_0,\xi} } w \right) d\xi \text{.}
\end{align}

Next, we introduce the vector fields and energies used in this paper.
We will use
\begin{align}
Z=(\partial_{\xi},\partial_{\eta})
\end{align}
as the commuting vector fields.
For a multi-index $a=(a_1,a_2)$, we set
\begin{align}
Z^{a}=\partial_{\xi}^{a_1}\partial_{\eta}^{a_2} \text{,} \qquad |a| = a_1 + a_2 \text{.}
\end{align}

Following \cite{MR3783412},
we will employ the weighted first-order energy
\begin{align}\label{based}
E_1(u(t))=\overline{E}_1(u(t))+\widehat{E}_1(u(t)) \text{,}
\end{align}
where
\begin{align}\label{based1}
\overline{E}_1(u(t))&=\|\langle \eta\rangle^{1+\delta}u_{\eta}\|_{L_{x}^2(\Sigma_t)}^2,\\
\widehat{E}_1(u(t))&=\|\langle \xi\rangle^{1+\delta}u_{\xi}\|_{L_{x}^2(\Sigma_t)}^2.
\end{align}
To denote the second-order energies, we use
\begin{align}
\overline{E}_2(u(t)) &= \sum_{|a|= 1}\overline{E}_1(Z^{a}u(t))=\|\langle \eta\rangle^{1+\delta}u_{\xi\eta}\|_{L_{x}^2(\Sigma_t)}^2+\|\langle \eta\rangle^{1+\delta}u_{\eta\eta}\|_{L_{x}^2(\Sigma_t)}^2 \text{,} \\
\widehat{{E}}_2(u(t)) &= \sum_{|a|= 1}\widehat{{E}}_1(Z^{a}u(t))=\|\langle \xi\rangle^{1+\delta}u_{\xi\xi}\|_{L_{x}^2(\Sigma_t)}^2+\|\langle \xi\rangle^{1+\delta}u_{\eta\xi}\|_{L_{x}^2(\Sigma_t)}^2 \text{,}
\end{align}
as well as
\begin{align}
E_2(u(t))=\sum_{|a|= 1}E_1(Z^{a}u(t))=\overline{E}_2(u(t))+\widehat{{E}}_2(u(t)) \text{.}
\end{align}

For the quasilinear case, we also need third-order energies.
We denote
\begin{align}
\label{Pe3fff8833399993} \overline{{E}}_3(u(t))&=\sum_{\substack{|a|=2\\a_2\neq 0}}\overline{E}_1(Z^{a}u(t))
=\|\langle \eta\rangle^{1+\delta}u_{\xi\eta\eta}\|_{L_{x}^2(\Sigma_t)}^2+\|\langle \eta\rangle^{1+\delta}u_{\eta\eta\eta}\|_{L_{x}^2(\Sigma_t)}^2 \text{,} \\
\label{Pe3ghyu338733} \widehat{{E}}_3(u(t))&=\sum_{\substack{|a|=2\\a_1\neq 0}}\widehat{{E}}_1(Z^{a}u(t))
=\|\langle \xi\rangle^{1+\delta}u_{\xi\eta\xi}\|_{L_{x}^2(\Sigma_t)}^2+\|\langle \xi\rangle^{1+\delta}u_{\xi\xi\xi}\|_{L_{x}^2(\Sigma_t)}^2 \text{,} \\
\label{Pe3fff883333} \widetilde{E}_3(u(t))&=\sum_{\substack{|a|=2\\a_2= 0}}\overline{{E}}_1(Z^{a}u(t))+\sum_{\substack{|a|=2\\a_1=0}}\widehat{E}_1(Z^{a}u(t))\nonumber\\
&=\|\langle \eta\rangle^{1+\delta}u_{\xi\xi\eta}\|_{L_{x}^2(\Sigma_t)}^2+\|\langle \xi\rangle^{1+\delta}u_{\eta\eta\xi}\|_{L_{x}^2(\Sigma_t)}^2 \text{,}
\end{align}
and
\begin{align}\label{ertyuhhhhhh}
E_3(u(t))=\overline{{E}}_3(u(t))+\widehat{{E}}_3(u(t))+\widetilde{E}_3(u(t)) \text{.}
\end{align}
Finally, we define the total energy,
\begin{align}\label{ertyuhhh9988jjhhh}
E(u(t))=E_1(u(t))+E_2(u(t))+E_3(u(t)) \text{.}
\end{align}

Next, we introduce the analogous spacetime weighted energies, which are inspired by \cite{Alinhac01} and \cite{MR2680391}.
First, we set
\begin{align}
\mathcal{E}_1(u(t))&=\overline{\mathcal{E}}_1(u(t))+\widehat{\mathcal{E}}_1(u(t)) \text{,} \\
\overline{\mathcal{E}}_1(u(t))&=\|\langle \xi\rangle^{-\frac{1+\delta}{2}}\langle \eta\rangle^{1+\delta}u_{\eta}\|^2_{L_{\tau,x}^2(S_t)} \text{,} \\
\widehat{\mathcal{E}}_1(u(t))&=\|\langle \eta\rangle^{-\frac{1+\delta}{2}}\langle \xi\rangle^{1+\delta}u_{\xi}\|^2_{L_{\tau,x}^2(S_t)} \text{.}
\end{align}
Similarly, we denote the second order spacetime weighted energies by
\begin{align}
\mathcal{E}_2(u(t))&=\overline{\mathcal{E}}_2(u(t))+\widehat{\mathcal{E}}_2(u(t)) \text{,} \\
\overline{\mathcal{E}}_2(u(t))&=\sum_{|a|=1}\overline{\mathcal{E}}_1(Z^au(t))\nonumber\\
&=\|\langle \xi\rangle^{-\frac{1+\delta}{2}}\langle \eta\rangle^{1+\delta}u_{\xi\eta}\|^2_{L_{\tau,x}^2(S_t)}+\|\langle \xi\rangle^{-\frac{1+\delta}{2}}\langle \eta\rangle^{1+\delta}u_{\eta\eta}\|^2_{L_{\tau,x}^2(S_t)},\\
\widehat{\mathcal{E}}_2(u(t))&=\sum_{|a|=1}\widehat{{\mathcal{E}}}_1(Z^au(t))\nonumber\\
&=\|\langle \eta\rangle^{-\frac{1+\delta}{2}}\langle \xi\rangle^{1+\delta}u_{\xi\xi}\|^2_{L_{\tau,x}^2(S_t)}+\|\langle \eta\rangle^{-\frac{1+\delta}{2}}\langle \xi\rangle^{1+\delta}u_{\eta\xi}\|^2_{L_{\tau,x}^2(S_t)}.
\end{align}

Again, we will require third-order energies for the quasilinear setting,
\begin{align}
\label{Pe3fff88333999930} \overline{\mathcal {E}}_3(u(t))&=\sum_{\substack{|a|=2\\a_2\neq  0}}\overline{\mathcal {E}}_1(Z^au(t))\nonumber\\
&=\|\langle \xi\rangle^{-\frac{1+\delta}{2}}\langle \eta\rangle^{1+\delta}u_{\xi\eta\eta}\|_{L_{\tau,x}^2(S_t)}^2+\|\langle \xi\rangle^{-\frac{1+\delta}{2}}\langle \eta\rangle^{1+\delta}u_{\eta\eta\eta}\|_{L_{\tau,x}^2(S_t)}^2 \text{,} \\
\label{Pe3ghyu3387330} \widehat{\mathcal {E}}_3(u(t))&=\sum_{\substack{|a|=2\\a_1\neq 0}}\widehat{\mathcal {E}}_1(Z^au(t))\nonumber\\
&=\|\langle \eta\rangle^{-\frac{1+\delta}{2}}\langle \xi\rangle^{1+\delta}u_{\xi\eta\xi}\|^2_{L^2_{\tau,x}(\Sigma_t)}+\|\langle \eta\rangle^{-\frac{1+\delta}{2}}\langle \xi\rangle^{1+\delta}u_{\xi\xi\xi}\|_{L_{\tau,x}^2(\Sigma_t)}^2 \text{,} \\
\label{Pe3fff8833330} \widetilde{\mathcal {E}}_3(u(t))&=\sum_{\substack{|a|=2\\a_2=  0}}\overline{\mathcal {E}}_1(Z^au(t))+\sum_{\substack{|a|=2\\a_1= 0}}\widehat{\mathcal {E}}_1(Z^au(t))\nonumber\\
&=\|\langle \xi\rangle^{-\frac{1+\delta}{2}}\langle \eta\rangle^{1+\delta}u_{\xi\xi\eta}\|_{L_{\tau,x}^2(S_t)}^2+\|\langle \eta\rangle^{-\frac{1+\delta}{2}}\langle \xi\rangle^{1+\delta}u_{\eta\eta\xi}\|_{L_{\tau,x}^2(S_t)}^2 \text{,}
\end{align}
and
\begin{align}\label{ertyuhhhhhh0}
\mathcal {E}_3(u(t))=\overline{\mathcal {E}}_3(u(t))+\widehat{\mathcal {E}}_3(u(t))+\widetilde{\mathcal {E}}_3(u(t)) \text{.}
\end{align}
Finally, we denote the total spacetime weighted energy by
\begin{align}\label{ertyuhhh9988jjhhh0}
\mathcal {E}(u(t))=\mathcal {E}_1(u(t))+\mathcal {E}_2(u(t))+\mathcal {E}_3(u(t)).
\end{align}

Fix $0<\delta<1$ (as in Theorem \ref{thn88786hhh9} or \ref{thn88786hhh9900}).
We then define the weights
\begin{align}\label{weight1}
\phi(x)=\langle x\rangle^{2+2\delta} \text{,} \qquad \psi(x)= \exp \left( -\int_{-\infty}^{x}\langle \rho\rangle^{-1-\delta}d\rho \right) \text{.}
\end{align}
Direct computations yield that
\begin{align}\label{weight2}
|\phi'(x)|\leq 4\langle x\rangle^{1+2\delta} \text{,} \qquad \psi'(x)=-\psi(x)\langle x\rangle^{-1-\delta} \text{.}
\end{align}
We also note that there exists a positive constant $c$ such that
\begin{align}\label{weight3}
c^{-1}\leq \psi(x)\leq 1 \text{,}
\end{align}
which then implies
\begin{align}\label{weight4}
c^{-1}\langle x\rangle^{-1-\delta}\leq -\psi'(x)\leq \langle x\rangle^{-1-\delta} \text{.}
\end{align}

\subsection{Some Estimates}

Here, we collect some basic estimates that will be useful throughout the proofs of our main theorems.
We let $f ( \xi )$ be a traveling wave, as in the setting of Theorem \ref{thn88786hhh9} or \ref{thn88786hhh9900}.

For convenience, we will make use of the following shorthands:
\begin{align}
\label{inviewof} \widetilde{F}(\xi,\eta)&=F(f'(\xi)+u_{\xi},u_{\eta}) \text{,} \\
\label{inviewof2} \widetilde{A}_i(\xi,\eta)&=A_i(f'(\xi)+u_{\xi},u_{\eta}) \text{,} \quad i=1,2,3 \text{.}
\end{align}
We will also use the notations
\begin{align}
\overline{A}_2(\xi,\eta)&=A_2(f'(\xi),u_{\eta}) \text{,} \\
\overline{\overline{A}}_2(\xi,\eta)&=\widetilde{A}_2(\xi,\eta)-\overline{A}_2(\xi,\eta) \text{.}
\end{align}

The next three lemmas are consequences of the fundamental theorem of calculus, the chain rule, and the Leibniz rule.
For the semilinear case (Theorem \ref{thn88786hhh9}), we need the following:

\begin{Lemma}\label{yinli111}
Assume $F$ satisfies \eqref{order4}, $f$ satisfies \eqref{fgty7888}, and
\begin{align}
\label{yinli111_ass} |u_{\xi}|+|u_{\eta}|\leq \nu_{0}
\end{align}
for some constant $\nu_0 > 0$.
Then, we have
 \begin{align}
 |\partial_{\xi}\widetilde{F}| &\leq C(|f{''}(\xi)||u_{\eta}|+|f{'}(\xi)||u_{\xi\eta}|+|u_{\eta}||u_{\xi\xi}|+|u_{\xi}||u_{\xi\eta}|) \text{,} \\
 |\partial_{\eta}\widetilde{F}| &\leq C(|f{'}(\xi)||u_{\eta\eta}|+|u_{\xi}||u_{\eta\eta}|+|u_{\eta}||u_{\xi\eta}|) \text{.}
 \end{align}
In particular, for any multi-index $a$, with $|a|\leq 1$, we have
 \begin{align}
 |Z^{a}\widetilde{F}| &\leq C(|Z^{a}f'(\xi)||u_{\eta}|+|f{'}(\xi)||Z^au_{\eta}|+|Z^au_{\xi}||u_{\eta}|+|u_{\xi}||Z^au_{\eta}|) \text{.}
 \end{align}
Here, the constants $C$ depend on $M_0$ and $\nu_0$.
\end{Lemma}

The following will be used in the quasilinear case (Theorem \ref{thn88786hhh9900}):

\begin{Lemma}\label{le55mmmad2222}
Assume $A_1$, $A_2$, $A_3$, and $F$ satisfy \eqref{order1}--\eqref{order4}; $f$ satisfies \eqref{fgty788899}; and
\begin{align}
\label{le55mmmad2222_ass} |u_{\xi}|+|u_{\eta}|+|u_{\xi\xi}|+|u_{\xi\eta}|+|u_{\eta\eta}|\leq \nu_1
\end{align}
for some constant $\nu_1 > 0$.
Then, we have
\begin{align}
|\widetilde{A}_1|&\leq C(|f'(\xi)|+|u_{\xi}|+|u_{\eta}|) \text{,}~~~~~~~|\widetilde{A}_2|\leq C(|f'(\xi)|+|u_{\xi}|) \text{,} \\
|\overline{A}_2|&\leq C|f'(\xi)| \text{,} \quad |\overline{\overline{A}}_2|\leq C|u_{\xi}| \text{,}~~~~~~|\widetilde{A}_3|\leq C|u_{\eta}| \text{,} \\
|\partial_{\xi}\widetilde{A}_1| & \leq C(|f''(\xi)|+|u_{\xi\xi}|+|u_{\xi\eta}|) \text{,}~~~~
|\partial_{\eta}\widetilde{A}_1| \leq C(|u_{\xi\eta}|+|u_{\eta\eta}|) \text{,} \\
|\partial_{\xi}\widetilde{A}_2| &\leq C(|f''(\xi)|+|u_{\xi\xi}|+|u_{\xi\eta}|) \text{,}~~~~
|\partial_{\eta}\widetilde{A}_2| \leq C(|f'(\xi)|+|u_{\xi}|+|u_{\xi\eta}|) \text{,} \\
|\partial_{\xi}\overline{A}_2| &\leq C(|f'(\xi)|+|f''(\xi)|) \text{,}~~~~~~~~~~~
|\partial_{\eta}\overline{A}_2| \leq C|f'(\xi)| \text{,} \\
|\partial_{\xi}\widetilde{A}_3| &\leq C(|u_{\eta}|+|u_{\xi\eta}|) \text{,}~~~~~~~~~~~~~~~~~
|\partial_{\eta}\widetilde{A}_3| \leq C(|u_{\eta}|+|u_{\eta\eta}|) \text{,} \\
|\partial_{\xi}\widetilde{F}| &\leq C(|f{''}(\xi)||u_{\eta}|+|f{'}(\xi)||u_{\xi\eta}|+|u_{\eta}||u_{\xi\xi}|+|u_{\xi}||u_{\xi\eta}|) \text{,} \\
|\partial_{\eta}\widetilde{F}| &\leq C(|f{'}(\xi)||u_{\eta\eta}|+|u_{\xi}||u_{\eta\eta}|+|u_{\eta}||u_{\xi\eta}|) \text{,}
\end{align}
where the constants $C$ depend on $M_1$ and $\nu_1$.
\end{Lemma}

\begin{Lemma}\label{xhjk999k99}
Assume $A_1$, $A_2$, $A_3$, and $F$ satisfy \eqref{order1}--\eqref{order4}; $f$ satisfies \eqref{fgty788899}; and
\begin{align}
\label{xhjk999k99_ass} |u_{\xi}|+|u_{\eta}|+|u_{\xi\xi}|+|u_{\xi\eta}|+|u_{\eta\eta}|\leq \nu_1
\end{align}
for some constant $\nu_1 > 0$.
Then for any multi-index $a$, with $|a|= 2$, we have
\begin{align}
| Z^a \widetilde{A}_1| &\leq C\big(|f''(\xi)|+|f^{(3)}(\xi)|\big)+C\sum_{|b|\leq 2}\big(|Z^{b}u_{\xi}|+|Z^{b}u_{\eta}|\big) \text{,} \\
|Z^a\widetilde{A}_2| &\leq C|f'(\xi)||u_{\eta\eta\eta}|+C\big(|f'(\xi)|+|f''(\xi)|+|f^{(3)}(\xi)|\big)\nonumber\\
&\qquad + C\sum_{|b|\leq 2}|Z^{b}u_{\xi}|+C|u_{\xi}|\sum_{|b|\leq 2}|Z^{b}u_{\eta}| \text{,} \\
|Z^a\widetilde{A}_3| &\leq C\sum_{|b|\leq 2}|Z^{b}u_{\eta}|+C|u_{\eta}|\sum_{|b|\leq 2}|Z^{b}u_{\xi}| \text{,} \\
|\partial^2_{\xi}\widetilde{F}| &\leq C\big(|f'(\xi)||u_{\xi\xi\eta}|+|f''(\xi)||u_{\xi\eta}|+|f''(\xi)||u_{\eta}|+|f^{(3)}(\xi)||u_{\eta}|\big)\nonumber\\
&\qquad +C\sum_{|b|+|c|\leq 2}|Z^{b}u_{\xi}||Z^{c}u_{\eta}| \text{,} \\
|\partial^2_{\eta}\widetilde{F}| &\leq C|f'(\xi)||u_{\eta\eta\eta}|+C|f'(\xi)||u_{\eta}|+C\sum_{|b|+|c|\leq 2}|Z^{b}u_{\xi}||Z^{c}u_{\eta}| \text{,} \\
|\partial^2_{\xi\eta}\widetilde{F}| &\leq C|f'(\xi)||u_{\xi\eta\eta}|+C|f''(\xi)||u_{\eta\eta}|+C\sum_{|b|+|c|\leq 2}|Z^{b}u_{\xi}||Z^{c}u_{\eta}| \text{,}
\end{align}
where the constants $C$ depend on $M_1$ and $\nu_1$.
\end{Lemma}

We will make use of the following Gronwall inequality, first established in \cite{Bellman}, in order to control the linear parts of our systems:

\begin{Lemma}\label{gronwall}
Assume that $h, \alpha$ and $\beta$ are nonnegative and smooth functions on $\mathbb{R}$, with sufficient decay at infinity.
If $h$ satisfies
\begin{align}
h(\xi_0)\leq \alpha(\xi_0)+\int_{-\infty}^{\xi_0}\beta(\xi)h(\xi)d\xi \text{,}
\end{align}
for each $\xi_0 \in \mathbb{R}$, then the following also holds for any $\xi_0 \in \mathbb{R}$:
\begin{align}
h(\xi_0)\leq \alpha(\xi_0)+\int_{-\infty}^{\xi_0}\alpha(\xi)\beta(\xi) \exp \bigg( \int_{\xi}^{\xi_0}\beta(\rho)d\rho \bigg) d\xi.
\end{align}
\end{Lemma}

The following pointwise estimates will be used frequently in the sequel:

\begin{Lemma}\label{xuyao8DD8899}
Let $u$ be a smooth function with sufficient decay at spatial infinity.
Then,
\begin{align}\label{xuyao217}
\|\langle\xi\rangle^{1+\delta}u_{\xi}\|_{L_{x}^{{\infty}}(\Sigma_t)}+\|\langle \eta\rangle^{1+\delta}u_{\eta}\|_{L_{x}^{{\infty}}(\Sigma_t)}&\leq C
E_1^{1/2}(u(t))+C E_2^{1/2}(u(t)) \text{,}
\end{align}
and
\begin{align}\label{xuyao219}
&\|\langle\eta\rangle^{-\frac{1+\delta}{2}}\langle\xi\rangle^{1+\delta}u_{\xi}\|_{L^2_{\tau}L_{x}^{{\infty}}(S_t)}+\|\langle\xi\rangle^{-\frac{1+\delta}{2}}\langle \eta\rangle^{1+\delta}u_{\eta}\|_{L^2_{\tau}L_{x}^{{\infty}}(S_t)}\nonumber\\
&\quad \leq C \mathcal {E}_1^{1/2}(u(t))+C \mathcal {E}_2^{1/2}(u(t)).
\end{align}
\end{Lemma}

\begin{proof}
These follow from the Sobolev embedding $H^1(\mathbb{R})\hookrightarrow L^{\infty}(\mathbb{R})$ and the estimates
\begin{align*}
|\partial_{x}\langle\xi\rangle^{1+\delta}|\leq C\langle\xi\rangle^{1+\delta} \text{,} &\qquad |\partial_{x}(\langle\eta\rangle^{-\frac{1+\delta}{2}}\langle\xi\rangle^{1+\delta})|\leq C \langle\eta\rangle^{-\frac{1+\delta}{2}}\langle\xi\rangle^{1+\delta} \text{,} \\
|\partial_{x}\langle\eta\rangle^{1+\delta}|\leq C\langle\eta\rangle^{1+\delta} \text{,} &\qquad |\partial_{x}(\langle\xi\rangle^{-\frac{1+\delta}{2}}\langle\eta\rangle^{1+\delta})|\leq C \langle\xi\rangle^{-\frac{1+\delta}{2}}\langle\eta\rangle^{1+\delta} \text{.} \qedhere
\end{align*}
 \end{proof}

The following bounds, which are similarly proved, are needed in the quasilinear case:

\begin{Lemma}\label{xu899quasilinear }
Let $u$ be a smooth function with sufficient decay at spatial infinity.
Then,
\begin{align}\label{xuyao217quasi}
\sum_{|a|\leq 1} \left[ \|\langle\xi\rangle^{1+\delta}Z^{a}u_{\xi}\|_{L_{x}^{{\infty}}(\Sigma_t)}+\|\langle \eta\rangle^{1+\delta}Z^{a}u_{\eta}\|_{L_{x}^{{\infty}}(\Sigma_t)} \right] &\leq C E^{1/2}(u(t)) \text{,}
\end{align}
and
\begin{align}\label{xuyao219quasi}
&\sum_{|a|\leq 1} \left[ \|\langle\eta\rangle^{-\frac{1+\delta}{2}}\langle\xi\rangle^{1+\delta}Z^{a}u_{\xi}\|_{L^2_{\tau}L_{x}^{{\infty}}(S_t)}+\|\langle\xi\rangle^{-\frac{1+\delta}{2}}\langle \eta\rangle^{1+\delta}Z^{a}u_{\eta}\|_{L^2_{\tau}L_{x}^{{\infty}}(S_t)} \right] \nonumber\\
&\quad \leq C \mathcal {E}^{1/2}(u(t)) \text{.}
\end{align}
 \end{Lemma}

\section{Proof of Theorem \ref{thn88786hhh9}}\label{sec3}

Assume in this section the hypotheses of Theorem \ref{thn88786hhh9}.
By standard local well-posedness and continuation results for semilinear wave equations (see, e.g., \cite{hor:lnheq}) and the relations \eqref{eq.uv_vf}, it suffices to show that our solution $u$ to \eqref{vsystem999} and \eqref{ddffk789yh} satisfies
\begin{equation}
\label{semilinear_goal} E_1 ( u (t) ) + E_2 ( u (t) ) \leq C^\ast \text{,}
\end{equation}
for a fixed constant $C^\ast > 0$ and for all $t > 0$ in the lifespan of $u$.

For this, we apply a bootstrap argument.
More specifically, in the remainder of this section, we assume $u$ is such a small solution on the time interval $[0, T]$, and that
\begin{align}\label{dddddd}
\sup_{0\leq t\leq T} [ E_1(u(t))+E_2(u(t))+\mathcal {E}_1(u(t))+\mathcal {E}_2(u(t)) ] \leq 4 A^2 \varepsilon^2 \text{,}
\end{align}
for a constant $A$, to be chosen later.
Our goal is then to establish the improved estimate
 \begin{align}\label{rgty6788888}
\sup_{0\leq t\leq T} [ E_1(u(t))+E_2(u(t))+\mathcal {E}_1(u(t))+\mathcal {E}_2(u(t)) ] \leq A^2 \varepsilon^2 \text{.}
\end{align}
The desired \eqref{semilinear_goal} then follows immediately from the above.

Note the assumption \eqref{small_semilinear}, together with \eqref{eq.uv_vf} and the system \eqref{vsystem999}, implies
\begin{equation}
\label{semilinear_id} E_1(u(0)) + E_2(u(0)) \leq C_0 \varepsilon^2 \text{,}
\end{equation}
for some constant $C_0 > 0$.
Furthermore, Lemma \ref{xuyao8DD8899} and \eqref{dddddd} imply that \eqref{yinli111_ass} holds, with $\nu_0 \simeq \varepsilon$, and hence the conclusions of Lemma \ref{yinli111} hold throughout this section.

\subsection{Energy Estimates}

In view of \eqref{inviewof}, the system \eqref{vsystem999} can be written as
\begin{align}
u_{\xi\eta}=\widetilde{F}.
\end{align}
For an multi-index $a$ with $|a|=0, 1$, we have
\begin{align}\label{df56666444}
Z^{a}u_{\xi\eta}=Z^{a}\widetilde{F}.
\end{align}

Fix $t_0 \in [0,T]$.
We first bound $\overline{E}_1(u(t_0))$, $\overline{E}_2(u(t_0))$, $\overline{\mathcal {E}}_1(u(t_0))$, and $\overline{\mathcal {E}}_2(u(t_0))$ using the energy method from \cite{MR3783412}, along with the Gronwall inequality for the linear part.
Multiplying both sides of \eqref{df56666444} by $2\langle \eta \rangle^{2+2\delta}Z^au^{T}_{\eta}$ and applying the Leibniz rule yields
\begin{align}\label{fghh6778yyy877}
(|\langle \eta \rangle^{1+\delta}Z^au_{\eta}|^2)_{\xi}=2\langle \eta \rangle^{2+2\delta}Z^au^{T}_{\eta}Z^{a}\widetilde{F} \text{.}
\end{align}
Fix, in addition, $\xi_0 \in \mathbb{R}$.
Integrating both sides of \eqref{fghh6778yyy877} over $S^{-}_{t_0,\xi_0}$, we then obtain
\begin{align}\label{23344455fg}
&\int_{\Sigma^{-}_{t_0,\xi_0}}|\langle \eta \rangle^{1+\delta}Z^au_{\eta}|^2 dx+{\sqrt{2}}\int_{C^{-}_{t_0,\xi_0}}|\langle \eta \rangle^{1+\delta}Z^au_{\eta}|^2\nonumber\\
&\quad = \int_{\Sigma^{-}_{0,\xi_0}}|\langle \eta \rangle^{1+\delta}Z^au_{\eta}|^2 dx + 2 \int_{S^{-}_{t_0,\xi_0}}\langle \eta \rangle^{2+2\delta}Z^au^{T}_{\eta}Z^{a}\widetilde{F} \text{.}
\end{align}

By Lemma \ref{yinli111}, \eqref{fgty7888}, and \eqref{key1},
\begin{align}
&\|\langle \eta \rangle^{2+2\delta}Z^au^{T}_{\eta}Z^{a}\widetilde{F}\|_{L^1_{t,x}(S^{-}_{t_0,\xi_0})}\nonumber\\
&\quad \leq C\big\|\langle \eta \rangle^{2+2\delta}|u_{\xi}||Z^au_{\eta}|^2\big\|_{L^1_{t,x}(S_{t_0})} + C\big\|\langle \eta \rangle^{2+2\delta}|u_{\eta}||Z^au_{\xi}||Z^{a}u_{\eta}|\big\|_{L^1_{t,x}(S_{t_0})} \nonumber \\
&\quad\qquad + C\big\||f'(\xi)|\langle \eta \rangle^{2+2\delta}|Z^{a}u_{\eta}|^2\big\|_{L^1_{t,x}(S^{-}_{t_0,\xi_0})} + C \big\||f''(\xi)|\langle \eta \rangle^{2+2\delta}|u_{\eta}||Z^{a}u_{\eta}|\big\|_{L^1_{t,x}(S^{-}_{t_0,\xi_0})} \nonumber\\
&\quad \leq C\|\langle \xi \rangle^{1+\delta} u_{\xi}\|_{L^{\infty}_{t,x}(S_{t_0})}\|\langle \xi \rangle^{-\frac{1+\delta}{2}}\langle \eta \rangle^{1+\delta}Z^au_{\eta}\|_{L^2_{t,x}(S_{t_0})}^2 \nonumber \\
&\quad \qquad + C \|\langle \xi \rangle^{-\frac{1+\delta}{2}}\langle \eta \rangle^{1+\delta}u_{\eta}\|_{L^{2}_{t}L^{\infty}_{x}(S_{t_0})}\|\langle \xi \rangle^{1+\delta}Z^au_{\xi}\|_{L^{\infty}_{t}L^2_{x}(S_{t_0})} \nonumber \\
&\quad \qquad \qquad \cdot \|\langle \xi \rangle^{-\frac{1+\delta}{2}}\langle \eta \rangle^{1+\delta}Z^{a}u_{\eta}\|_{L^2_{t,x}(S_{t_0})} \nonumber \\
&\quad \qquad + C\int_{-\infty}^{\xi_0}\langle \xi\rangle^{-1-\delta} \bigg[ \int_{C^{-}_{t_0,\xi}}(|\langle \eta \rangle^{1+\delta}Z^{a}u_{\eta}|^2+|\langle \eta \rangle^{1+\delta}u_{\eta}||\langle \eta \rangle^{1+\delta}Z^{a}u_{\eta}|) \bigg] d\xi \text{.}
\end{align}
Applying Lemma \ref{xuyao8DD8899} and summing over $a$, we then have
\begin{align}\label{fgghy67}
&\sum_{|a|\leq 1}\|\langle \eta \rangle^{2+2\delta}Z^au^{T}_{\eta}Z^{a}\widetilde{F}\|_{L^1_{t,x}(S^{-}_{t_0,\xi_0})} \nonumber \\
&\quad \leq C \sup_{0\leq t\leq t_0} \big[ E^{1/2}_1(u(t)) + E^{1/2}_2(u(t)) \big] [ \mathcal {E}_1(u(t_0)+\mathcal {E}_2(u(t_0)) ] \nonumber \\
&\quad \qquad + C \sum_{|a|\leq 1}\int_{-\infty}^{\xi_0}\langle \xi\rangle^{-1-\delta} \bigg( \int_{C^{-}_{t_0,\xi}}(|\langle \eta \rangle^{1+\delta} Z^a u_{\eta}|^2 \bigg) d\xi \text{.}
\end{align}

Next, from \eqref{semilinear_id}, \eqref{23344455fg}, and \eqref{fgghy67}, we see that
\begin{align}\label{23344455fffg}
&\sum_{|a|\leq 1}\int_{\Sigma^{-}_{t_0,\xi_0}}|\langle \eta \rangle^{1+\delta}Z^au_{\eta}|^2 dx+\sum_{|a|\leq 1}\int_{C^{-}_{t_0,\xi_0}}|\langle \eta \rangle^{1+\delta}Z^au_{\eta}|^2 \nonumber \\
&\quad \leq C \varepsilon^2 + C A^3 \varepsilon^3 + C \sum_{|a|\leq 1}\int_{-\infty}^{\xi_0}\langle \xi\rangle^{-1-\delta} \bigg( \int_{C^{-}_{t_0,\xi}}(|\langle \eta \rangle^{1+\delta}Z^{a}u_{\eta}|^2 \bigg) d\xi \text{.}
\end{align}
Then, it follows from the Gronwall inequality (Lemma \ref{gronwall}) that
\begin{align} \label{23344455ffghhhfg}
&\sum_{|a|\leq 1}\int_{\Sigma^{-}_{t_0,\xi_0}}|\langle \eta \rangle^{1+\delta}Z^au_{\eta}|^2 dx+\sum_{|a|\leq 1}\int_{C^{-}_{t_0,\xi_0}}|\langle \eta \rangle^{1+\delta}Z^au_{\eta}|^2\leq C\varepsilon^2+CA^3\varepsilon^3 \text{.}
\end{align}
Since $\xi_0$ is arbitrary, we conclude that
\begin{align}
\sum_{|a| \leq 1}\|\langle \eta\rangle^{1+\delta}Z^au_{\eta}\|_{L_{x}^2(\Sigma_{t_0})}^2\leq C\varepsilon^2+CA^3\varepsilon^3 \text{.}
\end{align}

Similarly, from \eqref{key2} and \eqref{23344455ffghhhfg}, we have
\begin{align}\label{fgggt5788}
\sum_{|a|\leq 1}\|\langle \xi\rangle^{-\frac{1+\delta}{2}}\langle \eta\rangle^{1+\delta}Z^au_{\eta}\|_{L_{t,x}^2(S_{t_0})}^2 &\leq C\sum_{|a|\leq 1}\int_{\mathbb{R}} \langle \xi_0\rangle^{-1-\delta} \left( \int_{C^{-}_{t_0,\xi_0}}|\langle \eta \rangle^{1+\delta}Z^au_{\eta}|^2 \right) d \xi_0 \nonumber \\
&\leq C\varepsilon^2+CA^3\varepsilon^3 \text{.}
\end{align}
Thus, combining \eqref{23344455ffghhhfg} and \eqref{fgggt5788} results in the energy bounds
\begin{align}\label{313}
\overline{E}_1(u(t_0))+\overline{E}_2(u(t_0))+\overline{\mathcal {E}}_1(u(t_0))+\overline{\mathcal {E}}_2(u(t_0))\leq C\varepsilon^2+CA^3\varepsilon^3.
\end{align}

We now turn to the estimates for $\widehat{E}_1(u(t_0)), \widehat{E}_2(u(t_0)), \widehat{\mathcal {E}}_1(u(t_0))$, and $\widehat{\mathcal {E}}_2(u(t_0))$.
We employ the energy approach in \cite{MR4098041}, and we control the linear part using on \eqref{313}.

Multiplying both sides of \eqref{df56666444} by $2\psi(\eta)\phi(\xi)Z^au^T_{\xi}$, we obtain
\begin{align}
(\psi(\eta)\phi(\xi)|Z^au_{\xi}|^2)_{\eta}-\psi'(\eta)\phi(\xi)|Z^au_{\xi}|^2
=2\psi(\eta)\phi(\xi)Z^au^T_{\xi}Z^a\widetilde{F} \text{.}
\end{align}
Integrating the above over $S_{ t_0 }$, we see that
\begin{align}\label{cxuai89yu99}
&\int_{\Sigma_{t_0}}\psi(\eta)\phi(\xi)|Z^au_{\xi}|^2dx-\int_{S_{t_0}}\psi'(\eta)\phi(\xi)|Z^au_{\xi}|^2 \nonumber \\
&\quad = \int_{\Sigma_{0}}\psi(\eta)\phi(\xi)|Z^au_{\xi}|^2dx +2\int_{S_{t_0}}\psi(\eta)\phi(\xi)Z^au^T_{\xi}Z^a\widetilde{F} \text{.}
\end{align}
In view of \eqref{cxuai89yu99}, \eqref{weight1}, \eqref{weight3}, \eqref{weight4}, and \eqref{semilinear_id}, we then obtain
\begin{align}\label{316}
&\widehat{E}_1(u(t_0))+\widehat{E}_2(u(t_0))+\widehat{\mathcal {E}}_1(u(t_0))+\widehat{\mathcal {E}}_2(u(t_0)) \nonumber \\
&\quad \leq C \varepsilon^2+C\sum_{|a|\leq 1}\|\langle \xi\rangle^{2+2\delta}Z^au^T_{\xi}Z^a\widetilde{F}\|_{L^1_{t,x}(S_{t_0})} \text{.}
\end{align}

By Lemma \ref{yinli111}, Lemma \ref{xuyao8DD8899}, and \eqref{fgty7888}, we estimate
\begin{align}\label{317}
&C \|\langle \xi\rangle^{2+2\delta}Z^au^T_{\xi}Z^a\widetilde{F}\|_{L^1_{t,x}(S_{t_0})} \nonumber \\
&\quad \leq C\big\|\langle \xi\rangle^{2+2\delta}|u_\eta||Z^au_{\xi}|^2\big\|_{L^1_{t,x}(S_{t_0})}+\big\|\langle \xi\rangle^{2+2\delta}|u_{\xi}||Z^au_{\xi}||Z^au_{\eta}|\big\|_{L^1_{t,x}(S_{t_0})} \nonumber \\
&\quad \qquad + C\big\||f'(\xi)|\langle \xi \rangle^{2+2\delta}|Z^{a}u_{\xi}||Z^{a}u_{\eta}|\big\|_{L^1_{t,x}(S_{t_0})}+C\big\||f''(\xi)|\langle \xi \rangle^{2+2\delta}|u_{\eta}||Z^{a}u_{\xi}|\big\|_{L^1_{t,x}(S_{t_0})} \nonumber \\
&\quad \leq C\|\langle \eta\rangle^{{1+\delta}}u_\eta\|_{{L^{\infty}_{t,x}(S_{t_0})}} \|\langle \eta\rangle^{-\frac{1+\delta}{2}}\langle \xi\rangle^{1+\delta}Z^au_{\xi}\|^2_{L^2_{t,x}(S_{t_0})} \nonumber \\
&\quad \qquad + C\|\langle \eta\rangle^{{1+\delta}}Z^au_\eta\|_{{L^{\infty}_{t}L^{2}_{x}(S_{t_0})}} \|\langle \eta\rangle^{-\frac{1+\delta}{2}}\langle \xi\rangle^{1+\delta}u_{\xi}\|_{L^2_{t}L^{\infty}_{x}(S_{t_0})} \nonumber \\
&\quad \qquad \qquad \cdot \|\langle \eta\rangle^{-\frac{1+\delta}{2}}\langle \xi\rangle^{1+\delta}Z^au_{\xi}\|_{L^2_{t,x}(S_{t_0})} \nonumber \\
&\quad \qquad +C\|\langle \xi\rangle^{\frac{3(1+\delta)}{2}}(|f'(\xi)|+|f''(\xi)|)\|_{L^{\infty}_{t,x}(S_{t_0})} \nonumber \\
&\quad \qquad \qquad \cdot \big\||\langle \eta\rangle^{-\frac{1+\delta}{2}}\langle \xi\rangle^{1+\delta}Z^{a}u_{\xi}\|_{L^2_{t,x}(S_{t_0})}\|\langle \xi\rangle^{-\frac{1+\delta}{2}}\langle \eta\rangle^{1+\delta}(|u_{\eta}|+|Z^{a}u_{\eta}|)\|_{L^2_{t,x}(S_{t_0})}\nonumber\\
&\quad \leq C \sup_{0\leq t\leq t_0} \big[ E^{1/2}_1(u(t)) + E^{1/2}_2(u(t)) \big] [ \mathcal {E}_1(u(t_0)+\mathcal {E}_2(u(t_0)) ] \nonumber \\
&\quad \qquad + C \big[ \overline{\mathcal {E}}_1(u(t_0)+\overline{\mathcal {E}}_2(u(t_0)) \big] + \frac{1}{100} \big[ \widehat{\mathcal {E}}_1(u(t_0))+\widehat{\mathcal {E}}_2(u(t_0)) \big] \text{.}
\end{align}
Combining \eqref{313}, \eqref{316}, and \eqref{317}, we then conclude
\begin{align}\label{318}
&\widehat{E}_1(u(t_0))+\widehat{E}_2(u(t_0))+\widehat{\mathcal {E}}_1(u(t_0))+\widehat{\mathcal {E}}_2(u(t_0))\leq C\varepsilon^2+CA^3\varepsilon^3 \text{.}
\end{align}

\subsection{Conclusion of the Proof}

From \eqref{313} and \eqref{318}, we have the estimate
\begin{align}
\sup_{0 \leq t\leq T} [ E_1(u(t))+E_2(u(t))+\mathcal {E}_1(u(t))+\mathcal {E}_2(u(t)) ] &\leq C_1 \varepsilon^2 + C_2 A^3 \varepsilon^3 \text{,}
\end{align}
for universal constants $C_1$ and $C_2$.
Taking $A^2 = 2 \max \{C_0, C_1\}$ and $\varepsilon_0$ sufficiently small such that $2 C_2 A \varepsilon_0 \leq 1$ results in the bound \eqref{rgty6788888} and completes the proof of Theorem \ref{thn88786hhh9}.

\section{Proof of Theorem \ref{thn88786hhh9900}} \label{sec4}

The key step in this proof is a bootstrap argument analogous to that in the previous section.
We assume the hypotheses of Theorem \ref{thn88786hhh9900}, and we let $u$ be a classical solution to \eqref{vsystem99999}, \eqref{ddffk789yh99} on $[0,T]$.
In addition, we assume the bound
\begin{align}\label{ddddde}
\sup_{0\leq t\leq T} [ E(u(t))+\mathcal {E}(u(t)) ] \leq 4 A^2 \varepsilon^2 \text{.}
\end{align}
Our key step will then be to prove, under the above assumptions, the improved estimate
 \begin{align}\label{rgty6788889}
\sup_{0\leq t\leq T} [ E(u(t))+\mathcal {E}(u(t)) ] \leq A^2 \varepsilon^2 \text{.}
\end{align}

First, note that \eqref{small_quasilinear}, \eqref{eq.uv_vf}, and \eqref{vsystem99999} imply
\begin{equation}
\label{quasilinear_id} E (u(0)) \leq C_0 \varepsilon^2 \text{,}
\end{equation}
for some constant $C_0 > 0$.
Also, Lemma \ref{xu899quasilinear } and \eqref{ddddde} imply \eqref{le55mmmad2222_ass} and \eqref{xhjk999k99_ass}, with $\nu_1 \simeq \varepsilon$, hence the conclusions of Lemmas \ref{le55mmmad2222} and \ref{xhjk999k99} hold throughout.
In addition, \eqref{hyper1}, \eqref{hyper3}, \eqref{inviewof2}, and the above imply that if $\varepsilon_0$ is sufficiently small, then for any $y \in \mathbb{R}^n$,
\begin{align} \label{hyper10}
y^{T} ( I - \widetilde{A}_1 - \widetilde{A}_2 ) y \geq \frac{\lambda}{2} |y|^2 \text{,} \qquad y^{T} ( I - \widetilde{A}_1 ) y \geq \frac{\lambda}{2} |y|^2 \text{.}
\end{align}

\subsection{Low-Order Energy Estimates}

We first estimate the lower-order (i.e., first and second-order) energies.
The strategy is similar to the semilinear case, except we must also deal with top-order terms.

In view of \eqref{inviewof} and \eqref{inviewof2}, the system \eqref{vsystem99999} can be written as
\begin{align}\label{noytingthat333}
u_{\xi\eta}=\widetilde{A}_1u_{\xi\eta}+\widetilde{A}_2u_{\eta\eta}+\widetilde{A}_3u_{\xi\xi}+\widetilde{A}_3f''(\xi)+\widetilde{F} \text{.}
\end{align}
For any multi-index $a=(a_1,a_2)$, with $|a|=1$, we see that $Z^{a}u$ satisfies
\begin{align}\label{quasiwave8989}
Z^{a}u_{\xi\eta}=\widetilde{A}_1Z^{a}u_{\xi\eta}+\widetilde{A}_2Z^{a}u_{\eta\eta} +G_a \text{,}
\end{align}
where
\begin{align}
G_{a} &= \widetilde{A}_3Z^{a}u_{\xi\xi} + Z^{a} \widetilde{A}_1 u_{\xi\eta} + Z^{a} \widetilde{A}_2 u_{\eta\eta} + Z^{a} \widetilde{A}_3 u_{\xi\xi} \nonumber \\
&\qquad + Z^{a} \widetilde{A}_3 f''(\xi)+\widetilde{A}_3Z^{a}f''(\xi)+Z^{a}\widetilde{F} \text{.}
\end{align}
It then follows from Lemma \ref{le55mmmad2222} and the above that
\begin{align}\label{xuyao88899934}
|G_a| \leq C \bigg[ |f'(\xi)|+|f''(\xi)|+|f^{(3)}(\xi)|+\sum_{|c|\leq 1}|Z^{c}u_{\xi}| \bigg] \sum_{|b|\leq 1}|Z^{b}u_{\eta}|+|u_{\eta}||Z^{a}u_{\xi\xi}| \text{.}
\end{align}

Fix $t_0 \in [0, T]$.
We first estimate $\overline{E}_1(u(t_0)), \overline{E}_2(u(t_0)), \overline{\mathcal {E}}_1(u(t_0))$ and $\overline{\mathcal {E}}_2(u(t_0))$.
Multiplying \eqref{quasiwave8989} by $2\phi(\eta)Z^au^{T}_{\eta}$ and noting the symmetry of $\widetilde{A}_1$ and $\widetilde{A}_2$, we obtain
\begin{align}\label{low11}
\big(\phi(\eta)|Z^au_{\eta}|^2\big)_{\xi} &= \big(\phi(\eta)Z^au^{T}_{\eta}\widetilde{A}_1Z^au_{\eta}\big)_{\xi} - \phi(\eta)Z^au^{T}_{\eta}\partial_{\xi}\widetilde{A}_1Z^au_{\eta}+\big(\phi(\eta)Z^au^{T}_{\eta}\widetilde{A}_2Z^au_{\eta}\big)_{\eta} \nonumber \\
&\quad \qquad - \phi'(\eta)Z^au^{T}_{\eta}\widetilde{A}_2Z^au_{\eta} - \phi(\eta)Z^au^{T}_{\eta}\partial_{\eta}\widetilde{A}_2Z^au_{\eta} + 2\phi(\eta)Z^au^{T}_{\eta}G_a \text{.}
\end{align}
Fixing $\xi_0 \in \mathbb{R}$ and integrating \eqref{low11} over $S^{-}_{t_0,\xi_0}$ then yields
\begin{align}\label{com1}
&\int_{\Sigma^{-}_{t_0,\xi_0}}e_0(t,x)dx+\sqrt{2}\int_{C^{-}_{t_0,\xi_0}}\widetilde{e}_0(t,x) \nonumber \\
&\quad = \int_{\Sigma^{-}_{0,\xi_0}}e_0(0,x) dx +\int_{S^{-}_{t_0,\xi_0}}q_0(t,x) +2\int_{S^{-}_{t_0,\xi_0}}\phi(\eta)Z^au^{T}_{\eta}G_a \text{,}
\end{align}
where
\begin{align}
{e}_0 &= \phi(\eta)Z^au^{T}_{\eta}(I-\widetilde{A}_1-\widetilde{A}_2)Z^au_{\eta} \text{,} \\
\widetilde{e}_0 &= \phi(\eta)Z^au^{T}_{\eta}(I-\widetilde{A}_1)Z^au_{\eta} \text{,} \\
q_0 &= -\phi(\eta)Z^au^{T}_{\eta}\partial_{\xi}\widetilde{A}_1Z^au_{\eta}-\phi'(\eta)Z^au^{T}_{\eta}\widetilde{A}_2Z^au_{\eta} - \phi(\eta)Z^au^{T}_{\eta}\partial_{\eta}\widetilde{A}_2Z^au_{\eta} \text{.}
\end{align}

In view of \eqref{hyper10} and \eqref{weight1}, we see that
\begin{align}\label{com20}
e_0(t,x) \geq \frac{\lambda}{2}|\langle \eta\rangle^{1+\delta} Z^au_{\eta}|^2 \text{,} \qquad \widetilde{e}_0 (t,x) \geq \frac{\lambda}{2}|\langle \eta\rangle^{1+\delta} Z^au_{\eta}|^2 \text{.}
\end{align}
By Lemma \ref{le55mmmad2222}, \eqref{fgty788899}, and Lemma \ref{xu899quasilinear }, we can bound
\begin{align}\label{com3}
|q_0(t,x)|&\leq C \bigg[ |f'(\xi)|+|f''(\xi)|+\sum_{|c|\leq 1}|Z^{c}u_{\xi}| \bigg] |\langle \eta\rangle^{1+\delta} Z^au_{\eta}|^2\nonumber\\
&\leq C\langle \xi\rangle^{-1-\delta}|\langle \eta\rangle^{1+\delta} Z^au_{\eta}|^2 \text{.}
\end{align}
Moreover, by \eqref{xuyao88899934}, \eqref{fgty788899}, and Lemma \ref{xu899quasilinear },
\begin{align}\label{com4}
|\phi(\eta)Z^au^{T}_{\eta}G_a| &\leq C \bigg[ |f'(\xi)|+|f''(\xi)|+|f^{(3)}(\xi)|+\sum_{|c|\leq 1}|Z^{c}u_{\xi}| \bigg] \sum_{|b|\leq 1}|\langle \eta\rangle^{1+\delta} Z^{b}u_{\eta}|^2 \nonumber \\
&\qquad + C\langle \eta \rangle^{2+2\delta}|Z^{a}u_{\eta}||u_{\eta}||Z^{a}u_{\xi\xi}| \nonumber \\
&\leq C\langle \xi\rangle^{-1-\delta}\sum_{|b|\leq 1}|\langle \eta\rangle^{1+\delta} Z^{b}u_{\eta}|^2+C\langle \eta \rangle^{2+2\delta}|Z^{a}u_{\eta}||u_{\eta}||Z^{a}u_{\xi\xi}| \text{.}
\end{align}

Thus, the combination of \eqref{com1}, \eqref{com20}--\eqref{com4}, and \eqref{quasilinear_id} gives
\begin{align}
&\sum_{|a| = 1} \int_{\Sigma^{-}_{t_0,\xi_0}}|\langle \eta\rangle^{1+\delta} Z^au_{\eta}|^2dx + \sum_{|a| = 1}\int_{C^{-}_{t_0,\xi_0}}|\langle \eta\rangle^{1+\delta} Z^au_{\eta}|^2 \nonumber \\
&\quad \leq C\varepsilon^2 +C\sum_{|a|\leq 1}\big\|\langle \eta \rangle^{2+2\delta}|Z^{a}u_{\eta}||u_{\eta}||Z^{a}u_{\xi\xi}|\big\|_{L^1_{t,x}(S_{t_0})} \nonumber \\
&\quad \qquad + C \sum_{|a|\leq 1}\int_{-\infty}^{\xi_0}\langle \xi\rangle^{-1-\delta} \bigg( \int_{C^{-}_{t_0,\xi}} |\langle \eta \rangle^{1+\delta}Z^{a}u_{\eta}|^2 \bigg) d\xi \text{.}
\end{align}
A similar process also yields analogous estimates with $|a| = 0$.
As a result, we conclude
\begin{align}\label{com1zui}
&\sum_{|a|\leq 1} \int_{\Sigma^{-}_{t_0,\xi_0}}|\langle \eta\rangle^{1+\delta} Z^au_{\eta}|^2dx+\sum_{|a|\leq 1}\int_{C^{-}_{t_0,\xi_0}}|\langle \eta\rangle^{1+\delta} Z^au_{\eta}|^2 \nonumber \\
&\quad \leq C\varepsilon^2 +C\sum_{|a|\leq 1}\big\|\langle \eta \rangle^{2+2\delta}|Z^{a}u_{\eta}||u_{\eta}||Z^{a}u_{\xi\xi}|\big\|_{L^1_{t,x}(S_{t_0})} \nonumber \\
&\quad \qquad + C \sum_{|a|\leq 1}\int_{-\infty}^{\xi_0}\langle \xi\rangle^{-1-\delta} \bigg( \int_{C^{-}_{t_0,\xi}} |\langle \eta \rangle^{1+\delta}Z^{a}u_{\eta}|^2 \bigg) d\xi \text{.}
\end{align}

By Lemma \ref{xu899quasilinear }, we bound
\begin{align}\label{ziizui}
&\big\||\langle \eta \rangle^{2+2\delta}Z^{a}u_{\eta}||u_{\eta}||Z^{a}u_{\xi\xi}|\big\|_{L^1_{t,x}(S_{t_0})} \nonumber \\
&\quad \leq \|\langle \xi\rangle^{-\frac{1+\delta}{2}}\langle \eta \rangle^{1+\delta}Z^{a}u_{\eta}\|_{L^2_{t,x}(S_{t_0})} \|\langle \xi\rangle^{-\frac{1+\delta}{2}}\langle \eta \rangle^{1+\delta}u_{\eta}\|_{L^2_{t}L_{x}^{\infty}(S_{t_0})} \|\langle \xi\rangle^{{1+\delta}}Z^{a}u_{\xi\xi}\|_{L^{\infty}_{t}L^2_{x}(S_{t_0})} \nonumber \\
&\quad \leq C \sup_{ 0 \leq t \leq t_0 } E^{1/2} (u(t)) \cdot \mathcal{E}(u(t_0)) \text{.}
\end{align}
Then, by \eqref{ddddde}, \eqref{com1zui}, \eqref{ziizui}, and the Gronwall inequality, we have
\begin{align}\label{com1zuidffggg}
&\sum_{|a|\leq 1}\int_{\Sigma^{-}_{t_0,\xi_0}}|\langle \eta\rangle^{1+\delta} Z^au_{\eta}|^2dx+\sum_{|a|\leq 1}\int_{C^{-}_{t_0,\xi_0}}|\langle \eta\rangle^{1+\delta} Z^au_{\eta}|^2\leq C\varepsilon^2 +CA^3\varepsilon^3 \text{.}
\end{align}
Arguing analogously to the semilinear case (see below \eqref{23344455ffghhhfg}), we see that \eqref{com1zuidffggg} yields
\begin{align}\label{31388899}
\overline{E}_1(u(t_0))+\overline{E}_2(u(t_0))+\overline{\mathcal {E}}_1(u(t_0))+\overline{\mathcal {E}}_2(u(t_0))\leq C\varepsilon^2+CA^3\varepsilon^3 \text{.}
\end{align}

Next, we turn to bounds for $\widehat{E}_1(u(t_0))$, $\widehat{E}_2(u(t_0))$, $\widehat{\mathcal {E}}_1(u(t_0))$, and $\widehat{\mathcal {E}}_2(u(t_0))$.
We write
\begin{align}\label{quasiwave8989juui}
Z^{a}u_{\xi\eta}=\widetilde{A}_1Z^{a}u_{\xi\eta}+\overline{{A}}_2Z^{a}u_{\eta\eta} + \overline{G}_a \text{,}
\end{align}
where again $|a| = 1$, and where
\begin{align}\label{yibar}
\overline{G}_{a}=\overline{\overline{{A}}}_2Z^{a}u_{\eta\eta}+G_a.
\end{align}
Multiplying \eqref{quasiwave8989juui} by $2\psi(\eta)\phi(\xi)Z^au^{T}_{\xi}$ and noting the symmetry of $\widetilde{A}_1$ and $\overline{{A}}_2$ yields
\begin{align}\label{yujiy8899}
&\big( \psi(\eta)\phi(\xi)|Z^au_{\xi}|^2 \big)_{\eta}-\psi'(\eta)\phi(\xi)|Z^au_{\xi}|^2 \nonumber \\
&\quad = \big(\psi(\eta)\phi(\xi)Z^au^{T}_{\xi}\widetilde{A}_1 Z^au_{\xi}\big)_{\eta}-\psi'(\eta)\phi(\xi)Z^au^{T}_{\xi}\widetilde{A}_1  Z^au_{\xi}-\psi(\eta)\phi(\xi)Z^au^{T}_{\xi}\partial_{\eta}\widetilde{A}_1  Z^au_{\xi} \nonumber \\
&\quad \qquad + 2\big(\psi(\eta)\phi(\xi)Z^au^{T}_{\xi}\overline{{A}}_2  Z^au_{\eta}\big)_{\eta}-2\psi'(\eta)\phi(\xi)Z^au^{T}_{\xi}\overline{{A}}_2  Z^au_{\eta} \nonumber \\
&\quad \qquad - 2\psi(\eta)\phi(\xi)Z^au^{T}_{\xi}\partial_{\eta}\overline{{A}}_2  Z^au_{\eta}-\big(\psi(\eta)\phi(\xi)Z^au^{T}_{\eta}\overline{{A}}_2  Z^au_{\eta}\big)_{\xi} \nonumber \\
&\quad \qquad + \psi(\eta)\phi'(\xi)Z^au^{T}_{\eta}\overline{{A}}_2  Z^au_{\eta} + \psi(\eta)\phi(\xi)Z^au^{T}_{\eta}\partial_{\xi}\overline{{A}}_2 Z^au_{\eta} \nonumber \\
&\quad \qquad + 2\psi(\eta)\phi(\xi)Z^au^{T}_{\xi}\overline{G}_a \text{.}
\end{align}

Integrating \eqref{yujiy8899} over $S_{t_0}$, it follows that
\begin{align}\label{itfolowjk99}
&\int_{\Sigma_{t_0}}e_1(t_0,x)dx+\int_{S_{t_0}}p_1(t,x) \nonumber \\
&\quad = \int_{\Sigma_{0}}e_1(0,x)dx+\int_{\Sigma_{t_0}}\widetilde{e}_1(t_0,x)dx-\int_{\Sigma_{0}}\widetilde{e}_1(0,x)dx \nonumber \\
&\quad \qquad +\int_{S_{t_0}}q_1(t,x)+2\int_{S_{t_0}}\psi(\eta)\phi(\xi)Z^au^{T}_{\xi}\overline{G}_a \text{,}
 \end{align}
where
\begin{align}
e_1&=\psi(\eta)\phi(\xi)Z^au^{T}_{\xi}(I-\widetilde{A}_1) Z^au_{\xi} \text{,} \\
p_1&=-\psi'(\eta)\phi(\xi)Z^au^{T}_{\xi}(I-\widetilde{A}_1)  Z^au_{\xi} \text{,} \\
\widetilde{e}_1&= 2\psi(\eta)\phi(\xi)Z^au^{T}_{\xi}\overline{{A}}_2 Z^au_{\eta}-\psi(\eta)\phi(\xi)Z^au^{T}_{\eta}\overline{{A}}_2  Z^au_{\eta},
\\
q_1&=-\psi(\eta)\phi(\xi)Z^au^{T}_{\xi}\partial_{\eta}\widetilde{A}_1 Z^au_{\xi}-2\psi'(\eta)\phi(\xi)Z^au^{T}_{\xi}\overline{{A}}_2 Z^au_{\eta} \nonumber \\
&\qquad - 2\psi(\eta)\phi(\xi)Z^au^{T}_{\xi}\partial_{\eta}\overline{{A}}_2  Z^au_{\eta} +\psi(\eta)\phi'(\xi)Z^au^{T}_{\eta}\overline{{A}}_2 Z^au_{\eta} \nonumber \\
&\qquad + \psi(\eta)\phi(\xi)Z^au^{T}_{\eta}\partial_{\xi}\overline{{A}}_2  Z^au_{\eta} \text{.}
\end{align}

In view of \eqref{hyper10}, \eqref{weight1}, \eqref{weight3}, and \eqref{weight4}, we have
\begin{align}\label{comm2}
e_1(t,x)\geq \frac{\lambda}{2c}|\langle \xi\rangle^{1+\delta} Z^au_{\xi}|^2 \text{,} \qquad p_1(t,x)\geq \frac{\lambda}{2c}|\langle \eta\rangle^{-\frac{1+\delta}{2}} \langle \xi\rangle^{1+\delta} Z^au_{\xi}|^2 \text{.}
\end{align}
Moreover, by Lemma \ref{le55mmmad2222} and \eqref{fgty788899},
\begin{align}\label{eeeert55e}
|\widetilde{e}_1(t,x)| &\leq C|f'(\xi)|\langle \xi\rangle^{2+2\delta}|Z^au_{\xi}||Z^au_{\eta}|+C|f'(\xi)|\langle \xi\rangle^{2+2\delta}|Z^au_{\eta}|^2 \nonumber \\
&\leq C|f'(\xi)|\langle \xi\rangle^{1+\delta}|\langle\xi\rangle^{1+\delta}Z^au_{\xi}||\langle \eta \rangle^{1+\delta}Z^au_{\eta}|+C|f'(\xi)|\langle \xi\rangle^{2+2\delta}|\langle \eta \rangle^{1+\delta}Z^au_{\eta}|^2 \nonumber \\
&\leq C|\langle\xi\rangle^{1+\delta}Z^au_{\xi}||\langle \eta \rangle^{1+\delta}Z^au_{\eta}|+C|\langle \eta \rangle^{1+\delta}Z^au_{\eta}|^2 \text{,}
\end{align}
and
\begin{align}\label{qonuu455}
|q_1(t,x)| &\leq C|f'(\xi)|\langle \xi\rangle^{2+2\delta}|Z^au_{\xi}||Z^au_{\eta}|+C\big(|f'(\xi)|+|f''(\xi)|\big)\langle \xi\rangle^{2+2\delta}|Z^au_{\eta}|^2 \nonumber \\
&\qquad + C\sum_{|b|\leq 1}|Z^{b}u_{\eta}||\langle \xi\rangle^{1+\delta}Z^{a}u_{\xi}|^2 \nonumber \\
&\leq C|f'(\xi)|\langle \xi\rangle^{\frac{3}{2}({1+\delta})}|\langle \eta\rangle^{-\frac{1+\delta}{2}}\langle \xi \rangle^{1+\delta}Z^au_{\xi}||\langle \xi\rangle^{-\frac{1+\delta}{2}}\langle \eta \rangle^{1+\delta}Z^au_{\eta}| \nonumber \\
&\qquad + C\big(|f'(\xi)|+|f''(\xi)|\big)\langle \xi\rangle^{3(1+\delta)}|\langle \xi\rangle^{-\frac{1+\delta}{2}}\langle \eta \rangle^{1+\delta}Z^au_{\eta}|^2 \nonumber \\
&\qquad + C\sum_{|b|\leq 1}|\langle \eta\rangle^{1+\delta}Z^{b}u_{\eta}||\langle \eta\rangle^{-\frac{1+\delta}{2}}\langle \xi\rangle^{1+\delta}Z^{a}u_{\xi}|^2 \nonumber \\
&\leq C|\langle \eta\rangle^{-\frac{1+\delta}{2}}\langle \xi \rangle^{1+\delta}Z^au_{\xi}||\langle \xi\rangle^{-\frac{1+\delta}{2}}\langle \eta \rangle^{1+\delta}Z^au_{\eta}|+C|\langle \xi\rangle^{-\frac{1+\delta}{2}}\langle \eta \rangle^{1+\delta}Z^au_{\eta}|^2 \nonumber \\
&\qquad + C\sum_{|b|\leq 1}|\langle \eta\rangle^{1+\delta}Z^{b}u_{\eta}||\langle \eta\rangle^{-\frac{1+\delta}{2}}\langle \xi\rangle^{1+\delta}Z^{a}u_{\xi}|^2 \text{.}
\end{align}

From \eqref{xuyao88899934}, \eqref{yibar}, Lemma \ref{le55mmmad2222}, and \eqref{fgty788899}, we get
\begin{align}\label{xuyao98790}
&|\psi(\eta)\phi(\xi)Z^au^{T}_{\xi}\overline{G}_a| \nonumber \\
&\quad \leq C\big(|f'(\xi)|+|f''(\xi)|+|f^{(3)}(\xi)|\big)\langle \xi\rangle^{2+2\delta}|Z^au_{\xi}|\sum_{|b|\leq 1}|Z^bu_{\eta}| \nonumber \\
&\quad \qquad + C\langle \xi\rangle^{2+2\delta}\sum_{|b|\leq 1}|Z^bu_{\eta}|\sum_{|c|\leq 1}|Z^cu_{\xi}|^2 + C\langle \xi\rangle^{2+2\delta}|Z^au_{\xi}||u_{\eta}||Z^au_{\xi\xi}| \nonumber \\
&\quad \qquad + C\langle \xi\rangle^{2+2\delta}|Z^au_{\xi}||u_{\xi}||Z^au_{\eta\eta}| \nonumber \\
&\quad \leq C|\langle \eta\rangle^{-\frac{1+\delta}{2}}\langle \xi \rangle^{1+\delta}Z^au_{\xi}|\sum_{|b|\leq 1}|\langle \xi\rangle^{-\frac{1+\delta}{2}}\langle \eta \rangle^{1+\delta}Z^bu_{\eta}| \nonumber \\
&\quad \qquad + C\sum_{|b|\leq 2}|\langle\eta\rangle^{1+\delta}Z^bu_{\eta}|\sum_{|c|\leq 2}|\langle \eta\rangle^{-\frac{1+\delta}{2}}\langle \xi \rangle^{1+\delta}Z^cu_{\xi}|\sum_{|d|\leq 1}|\langle \eta\rangle^{-\frac{1+\delta}{2}}\langle \xi \rangle^{1+\delta}Z^du_{\xi}| \text{.}
\end{align}
Thus, it follows from \eqref{itfolowjk99} and \eqref{comm2}--\eqref{xuyao98790} that for any multi-index $a$ with $|a| = 1$,
\begin{align}\label{fuzaio89}
&\|\langle \xi\rangle^{1+\delta} Z^au_{\xi}\|^2_{L^2_{x}(\Sigma_{t_0})}+\|\langle \eta\rangle^{-\frac{1+\delta}{2}} \langle \xi\rangle^{1+\delta} Z^au_{\xi}\|^2_{L^2_{t,x}(S_{t_0})} \nonumber \\
&\quad \leq C\varepsilon^2+C\|\langle\xi\rangle^{1+\delta}Z^au_{\xi}\|_{L^2_{x}(\Sigma_{t_0})} \|\langle \eta \rangle^{1+\delta}Z^au_{\eta}|\|_{L^2_{x}(\Sigma_{t_0})}+C\|\langle \eta \rangle^{1+\delta}Z^au_{\eta}\|^2_{L^2_{x}(\Sigma_{t_0})} \nonumber \\
&\quad \qquad + C\|\langle \eta\rangle^{-\frac{1+\delta}{2}}\langle \xi \rangle^{1+\delta}Z^au_{\xi}\|_{L^2_{t,x}(S_{t_0})}\sum_{|b|\leq 1}\|\langle \xi\rangle^{-\frac{1+\delta}{2}}\langle \eta \rangle^{1+\delta}Z^bu_{\eta}\|_{L^2_{t,x}(S_{t_0})} \nonumber \\
&\quad \qquad + C\|\langle \xi\rangle^{-\frac{1+\delta}{2}}\langle \eta \rangle^{1+\delta}Z^au_{\eta}\|^2_{L^2_{t,x}(S_{t_0})} \nonumber \\
&\quad \qquad + C\sum_{|b|\leq 2}\|\langle\eta\rangle^{1+\delta}Z^bu_{\eta}\|_{L^{\infty}_{t}L^{2}_{x}(S_{t_0})}\sum_{|c|\leq 2}\|\langle \eta\rangle^{-\frac{1+\delta}{2}}\langle \xi \rangle^{1+\delta}Z^cu_{\xi}\|_{L^2_{t,x}(S_{t_0})} \nonumber \\
&\quad \qquad \qquad \cdot \sum_{|d|\leq 1}\|\langle \eta\rangle^{-\frac{1+\delta}{2}}\langle \xi \rangle^{1+\delta}Z^du_{\xi}\|_{L^{2}_{t}L^{\infty}_{x}(S_{t_0})} \nonumber \\
&\quad \leq C\varepsilon^2+\frac{1}{100}\|\langle \xi\rangle^{1+\delta} Z^au_{\xi}\|^2_{L^2_{x}(\Sigma_{t_0})}+\frac{1}{100}\|\langle \eta\rangle^{-\frac{1+\delta}{2}} \langle \xi\rangle^{1+\delta} Z^au_{\xi}\|^2_{L^2_{t,x}(S_{t_0})} \nonumber \\
&\quad \qquad + C [ \overline{E}_1(u(t_0))+\overline{E}_2(u(t_0))+\overline{\mathcal {E}}_1(u(t_0))+\overline{\mathcal {E}}_2(u(t_0)) ] \nonumber \\
&\quad \qquad + C \sup_{ 0 \leq t \leq t_0 } E^{1/2}(u(t)) \cdot \mathcal {E}(u(t_0)) \text{.}
\end{align}

A similar (but easier) estimate yields the same bound \eqref{fuzaio89} also for $|a| = 0$.
Therefore, combining \eqref{31388899}, \eqref{fuzaio89}, and the above, we obtain the estimate
\begin{align}\label{xuyai90656}
\widehat{E}_1(u(t_0))+\widehat{E}_2(u(t_0))+\widehat{\mathcal {E}}_1(u(t_0))+\widehat{\mathcal {E}}_2(u(t_0))\leq C\varepsilon^2+CA^3\varepsilon^3 \text{.}
\end{align}
Finally, \eqref{31388899} and \eqref{xuyai90656} together yield
\begin{align}\label{duix998709}
\sup_{0 \leq t \leq T} [ E_1(u(t))+E_2(u(t))+\mathcal {E}_1(u(t))+\mathcal {E}_2(u(t)) ] &\leq C\varepsilon^2+CA^3\varepsilon^3 \text{.}
\end{align}

\subsection{Top-Order Energy Estimates}

We next turn to the top-order (i.e., third-order) energy estimates, which are also the key and most delicate parts in the analysis.
As usual, we fix $t_0 \in [ 0, T ]$ and $\xi_0 \in \mathbb{R}$.

\subsubsection{Estimates for $\overline{E}_3$ and $\overline{\mathcal{E}}_3$}

In contrast to the semilinear setting and the low-order estimates in the previous subsection, here we must estimate $\overline{{E}}_3(u(t_0))$ and $\overline{{{\mathcal{E}}}}_3(u(t_0))$ separately.

Let $a = (a_1, a_2)$ be a multi-index, with $|a|=2$ and $a_2\neq 0$.
Then, by \eqref{noytingthat333},
\begin{align}\label{gajie8989}
Z^{a}u_{\xi\eta}=\widetilde{A}_1Z^{a}u_{\xi\eta}+\widetilde{A}_2Z^{a}u_{\eta\eta}+\widetilde{A}_3Z^{a}u_{\xi\xi}+H_a \text{,}
\end{align}
where
\begin{align}
H_a &= \sum_{\substack{b+c=a\\c\neq a}}\lambda_{bc}\big(Z^{b}\widetilde{A}_1Z^{c}u_{\xi\eta}+Z^{b}\widetilde{A}_2Z^{c}u_{\eta\eta}+Z^{b}\widetilde{A}_3Z^{c}u_{\xi\xi}\big) \nonumber \\
&\qquad + \sum_{{b+c=a}}\lambda_{bc}Z^{b}\widetilde{A}_3Z^{c}f''(\xi) +Z^{a}\widetilde{F} \text{,}
\end{align}
and where the $\lambda_{bc}$'s are integer constants.
By Lemmas \ref{le55mmmad2222} and \ref{xhjk999k99},
\begin{align}\label{xuyao9jllj566}
|H_a| &\leq C [ |f'(\xi)|+|f''(\xi)|+|f^{(3)}(\xi)| ] \bigg[ \sum_{\substack{|b|=2\\b_2\neq 0}}|Z^{b}u_{\eta}|+\sum_{|b|\leq 1}|Z^{b}u_{\eta}| \bigg] \nonumber \\
&\qquad + C\sum_{|b|\leq 1}|Z^{b}u_{\eta}|\sum_{|c|\leq 2}|Z^{b}u_{\xi}|+C\sum_{|b|\leq 2}|Z^{b}u_{\eta}|\sum_{|c|\leq 1}|Z^{b}u_{\xi}| \text{.}
\end{align}

Multiplying \eqref{gajie8989} by $2\phi(\eta)Z^au^{T}_{\eta}$ and recalling that $\widetilde{A}_1$, $\widetilde{A}_2$, $\widetilde{A}_3$ are symmetric, we obtain
\begin{align}\label{highgh7811}
\big(\phi(\eta)|Z^au_{\eta}|^2\big)_{\xi} &= \big(\phi(\eta)Z^au^{T}_{\eta}\widetilde{A}_1Z^au_{\eta}\big)_{\xi} - \phi(\eta)Z^au^{T}_{\eta}\partial_{\xi}\widetilde{A}_1Z^au_{\eta} + \big(\phi(\eta)Z^au^{T}_{\eta}\widetilde{A}_2Z^au_{\eta}\big)_{\eta} \nonumber \\
&\qquad - \phi'(\eta)Z^au^{T}_{\eta}\widetilde{A}_2Z^au_{\eta} - \phi(\eta)Z^au^{T}_{\eta}\partial_{\eta}\widetilde{A}_2Z^au_{\eta} + 2\big(\phi(\eta)Z^au^{T}_{\eta}\widetilde{A}_3Z^au_{\xi}\big)_{\xi} \nonumber \\
&\qquad - 2\phi(\eta)Z^au^{T}_{\eta}\partial_{\xi}\widetilde{A}_3Z^au_{\xi} - \big(\phi(\eta)Z^au^{T}_{\xi}\widetilde{A}_3Z^au_{\xi}\big)_{\eta} + \phi'(\eta)Z^au^{T}_{\xi}\widetilde{A}_3Z^au_{\xi} \nonumber \\
&\qquad + \phi(\eta)Z^au^{T}_{\xi}\partial_{\eta}\widetilde{A}_3Z^au_{\xi} + 2\phi(\eta)Z^au^{T}_{\eta}H_a \text{.}
\end{align}
Integrating both sides of \eqref{highgh7811} over $S^{-}_{t_0,\xi_0}$, we then have
\begin{align}\label{com88991}
&\int_{\Sigma^{-}_{t_0,\xi_0}}e_2(t,x)dx + \sqrt{2}\int_{C^{-}_{t_0,\xi_0}}\overline{e}_2(t,x) \nonumber \\
&\quad = \int_{\Sigma^{-}_{0,\xi_0}}e_2(0,x) dx + \int_{\Sigma^{-}_{t_0,\xi_0}}\overline{\overline{e}}_2(t,x)dx - \int_{\Sigma^{-}_{0,\xi_0}}\overline{\overline{e}}_2(0,x) dx + \sqrt{2}\int_{C^{-}_{t_0,\xi_0}}\widetilde{e}_2(t,x) \nonumber \\
&\quad \qquad + \int_{S^{-}_{t_0,\xi_0}}q_2(t,x) + 2\int_{S^{-}_{t_0,\xi_0}}\phi(\eta)Z^au^{T}_{\eta}H_a \text{,}
\end{align}
where
\begin{align}
e_2 &= \phi(\eta)Z^au^{T}_{\eta}(I-\widetilde{A}_1-\widetilde{A}_2)Z^au_{\eta} \text{,} \\
\overline{e}_2 &= \phi(\eta)Z^au^{T}_{\eta}(I-\widetilde{A}_1)Z^au_{\eta} \text{,} \\
\overline{\overline{e}}_2 &= 2\phi(\eta)Z^au^{T}_{\eta}\widetilde{A}_3Z^au_{\xi} - \phi(\eta)Z^au^{T}_{\xi}\widetilde{A}_3Z^au_{\xi} \text{,} \\
\widetilde{e}_2 &= 2 \phi(\eta)Z^au^{T}_{\eta}\widetilde{A}_3Z^au_{\xi} \text{,} \\
q_2 &= - \phi(\eta)Z^au^{T}_{\eta}\partial_{\xi}\widetilde{A}_1Z^au_{\eta}-\phi'(\eta)Z^au^{T}_{\eta}\widetilde{A}_2Z^au_{\eta} - \phi(\eta)Z^au^{T}_{\eta}\partial_{\eta}\widetilde{A}_2Z^au_{\eta} \nonumber \\
&\qquad - 2\phi(\eta)Z^au^{T}_{\eta}\partial_{\xi}\widetilde{A}_3Z^au_{\xi} + \phi'(\eta)Z^au^{T}_{\xi}\widetilde{A}_3Z^au_{\xi}+\phi(\eta)Z^au^{T}_{\xi}\partial_{\eta}\widetilde{A}_3Z^au_{\xi} \text{.}
\end{align}

From \eqref{hyper10}, we immediately see that
\begin{align}\label{cjjjjom2}
e_2(t,x) \geq \frac{\lambda}{2}|\langle \eta\rangle^{1+\delta} Z^au_{\eta}|^2 \text{,} \qquad \overline{e}_2(t,x)\geq \frac{\lambda}{2}|\langle \eta\rangle^{1+\delta} Z^au_{\eta}|^2 \text{.}
\end{align}
Moreover, by Lemma \ref{le55mmmad2222},
\begin{align}\label{eeejccjjjert55e999}
&|\widetilde{e}_2(t,x)|\leq C|\langle \eta\rangle^{1+\delta}u_{\eta}||\langle \eta\rangle^{1+\delta}Z^au_{\eta}||Z^au_{\xi}| \text{,}
\end{align}
and
\begin{align}\label{eeejccjjjert55e}
|\overline{\overline{{e}}}_2(t,x)| &\leq C\langle \eta\rangle^{2+2\delta}|u_{\eta}||Z^au_{\xi}|^2 + C |\langle \eta\rangle^{1+\delta}u_{\eta}||\langle \eta\rangle^{1+\delta}Z^au_{\eta}||Z^au_{\xi}| \nonumber \\
&\leq C |\langle \eta\rangle^{1+\delta}u_{\eta}|\sum_{|b|=2}|\langle \eta\rangle^{1+\delta}Z^bu_{\eta}||Z^au_{\xi}| \text{.}
\end{align}
Here, we also made use the simple observation that since $a_2\neq 0$,
\begin{align}\label{fact}
| Z^{a}u_{\xi} | = |\partial_{\xi}^{a_1}\partial_{\eta}^{a_2}u_{\xi}|\leq \sum_{|b|=2}|Z^{b}u_{\eta}| \text{.}
\end{align}
By Lemma \ref{le55mmmad2222}, \eqref{fgty788899}, \eqref{fact}, and Sobolev embedding,
\begin{align}\label{eeejffffccjjjert55e}
|q_2(t,x)| &\leq C\big(|f'(\xi)|+|f''(\xi)|\big) |\langle \eta\rangle^{1+\delta}Z^{a}u_{\eta}|^2+C|\langle \eta\rangle^{1+\delta}Z^{a}u_{\eta}|^2(|u_{\xi}|+|u_{\xi\xi}|+|u_{\xi\eta}|) \nonumber \\
&\qquad + C \langle \eta\rangle^{2+2\delta}|Z^{a}u_{\eta}||Z^{a}u_{\xi}|(|u_{\eta}|+|u_{\xi\eta}|)+C\langle \eta\rangle^{2+2\delta}|Z^{a}u_{\xi}|^2(|u_{\eta}|+|u_{\eta\eta}|) \nonumber \\
&\leq C \langle \xi\rangle^{-1-\delta}|\langle \eta\rangle^{1+\delta}Z^{a}u_{\eta}|^2 \nonumber \\
&\qquad + C|\langle\xi\rangle^{1+\delta}Z^{a}u_{\xi}|\sum_{|b|= 2}|\langle\xi\rangle^{-\frac{1+\delta}{2}}\langle \eta\rangle^{1+\delta}Z^{b}u_{\eta}|\sum_{|c|\leq 1}|\langle\xi\rangle^{-\frac{1+\delta}{2}}\langle \eta\rangle^{1+\delta}Z^{c}u_{\eta}| \text{.}
\end{align}
Moreover, it follows from \eqref{xuyao9jllj566} and \eqref{fgty788899} that
\begin{align}\label{esuiyujjanjjert55e}
&\sum_{\substack{|a|=2\\a_2\neq 0}}|\phi(\eta)Z^au^{T}_{\eta}H_a| \nonumber \\
&\quad \leq C \langle \xi\rangle^{-1-\delta}\sum_{\substack{|a|=2\\a_2\neq 0}}|\langle \eta\rangle^{1+\delta}Z^au_{\eta}|^2+C \sum_{|b|\leq 1} |\langle \xi\rangle^{-\frac{1+\delta}{2}}\langle \eta\rangle^{1+\delta}Z^bu_{\eta}|^2 \nonumber \\
&\quad \qquad + C \sum_{|a|\leq 2}|\langle\xi\rangle^{-\frac{1+\delta}{2}}\langle \eta\rangle^{1+\delta}Z^{a}u_{\eta}|\sum_{|b|\leq 1}|\langle\xi\rangle^{-\frac{1+\delta}{2}}\langle \eta\rangle^{1+\delta}Z^{b}u_{\eta}|\sum_{|c|\leq 2}|\langle\xi\rangle^{1+\delta}Z^{c}u_{\xi}| \nonumber \\
&\quad \qquad + C \sum_{|b|\leq 2}|\langle\xi\rangle^{-\frac{1+\delta}{2}}\langle \eta\rangle^{1+\delta}Z^{b}u_{\eta}|^2\sum_{|c|\leq 1}|\langle\xi\rangle^{1+\delta}Z^{c}u_{\xi}| \text{.}
\end{align}

Now, from \eqref{com88991} and \eqref{cjjjjom2}, we obtain
\begin{align}\label{cfffom88991}
&\sum_{\substack{|a|=2\\a_2\neq 0}} \int_{\Sigma^{-}_{t_0,\xi_0}}|\langle \eta\rangle^{1+\delta} Z^au_{\eta}|^2dx+\sum_{\substack{|a|=2\\a_2\neq 0}}\int_{C^{-}_{t_0,\xi_0}}|\langle \eta\rangle^{1+\delta} Z^au_{\eta}|^2 \nonumber \\
&\quad \leq C\varepsilon^2+C\sum_{\substack{|a|=2\\a_2\neq 0}} \|\overline{\overline{{e}}}_2\|_{L^1_{x}(\Sigma_{t_0})}+C\sum_{\substack{|a|=2\\a_2\neq 0}}\int_{C^{-}_{t_0,\xi_0}}|\widetilde{e}_2(t,x)| \nonumber \\
&\quad \qquad + C\sum_{\substack{|a|=2\\a_2\neq 0}}\int_{S^{-}_{t_0,\xi_0}}|q_2(t,x)| +C\sum_{\substack{|a|=2\\a_2\neq 0}}\int_{S^{-}_{t_0,\xi_0}}|\phi(\eta)Z^au^{T}_{\eta}H_a| \text{.}
\end{align}
By \eqref{eeejccjjjert55e} and Lemma \ref{xu899quasilinear },
\begin{align}\label{eexuyao88ejccjjjert55e}
\|\overline{\overline{{e}}}_2\|_{L^1_{x}(\Sigma_{t_0})} &\leq C\|\langle \eta\rangle^{1+\delta}u_{\eta}\|_{L^{\infty}_{x}(\Sigma_{t_0})}\|Z^au_{\xi}\|_{L^2_{x}(\Sigma_{t_0})} \sum_{|b|=2}\|\langle \eta\rangle^{1+\delta}Z^bu_{\eta}\|_{L^2_{x}(\Sigma_{t_0})} \nonumber \\
&\leq C E^{3/2}(u(t_0)) \text{.}
\end{align}
In addition, from \eqref{eeejffffccjjjert55e}, \eqref{esuiyujjanjjert55e}, \eqref{key1} and Lemma \ref{xu899quasilinear }, we obtain
\begin{align}\label{jkui89ghk}
&\sum_{\substack{|a|=2\\a_2\neq 0}}\int_{S^{-}_{t_0,\xi_0}}|q_2(t,x)| +\sum_{\substack{|a|=2\\a_2\neq 0}}\int_{S^{-}_{t_0,\xi_0}}|\phi(\eta)Z^au^{T}_{\eta}H_a| \nonumber \\
&\quad \leq C \sum_{\substack{|a|=2\\a_2\neq 0}} \int_{-\infty}^{\xi_0}\langle \xi\rangle^{-1-\delta} \bigg( \int_{C^{-}_{t_0,\xi}}|\langle \eta\rangle^{1+\delta}Z^au_{\eta}|^2 \bigg) d\xi + C \sum_{|b|\leq 1} \|\langle \xi\rangle^{-\frac{1+\delta}{2}}\langle \eta\rangle^{1+\delta}Z^bu_{\eta}\|_{L_{t,x}^2({S_{t_0}})}^2 \nonumber \\
&\quad \qquad + C\sum_{|a|\leq 2}\|\langle\xi\rangle^{-\frac{1+\delta}{2}}\langle \eta\rangle^{1+\delta}Z^{a}u_{\eta}\|_{L_{t,x}^2({S_{t_0}})}\sum_{|b|\leq 1}\|\langle\xi\rangle^{-\frac{1+\delta}{2}}\langle \eta\rangle^{1+\delta}Z^{b}u_{\eta}\|_{L_{t}^2L_{x}^{\infty}({S_{t_0}})} \nonumber \\
&\quad \qquad \qquad \cdot \sum_{|c|\leq 2}\|\langle\xi\rangle^{1+\delta}Z^{c}u_{\xi}\|_{L_{t}^{\infty}L_{x}^{2}({S_{t_0}})} \nonumber \\
&\quad \qquad + C\sum_{|b|\leq 2}\|\langle\xi\rangle^{-\frac{1+\delta}{2}}\langle \eta\rangle^{1+\delta}Z^{b}u_{\eta}\|_{L_{t,x}^2({S_{t_0}})}^2\sum_{|c|\leq 1}\|\langle\xi\rangle^{1+\delta}Z^{c}u_{\xi}\|_{L_{t,x}^{\infty}({S_{t_0}})} \nonumber \\
&\quad \leq C \sum_{\substack{|a|=2\\a_2\neq 0}} \int_{-\infty}^{\xi_0}\langle \xi\rangle^{-1-\delta} \bigg( \int_{C^{-}_{t_0,\xi}}|\langle \eta\rangle^{1+\delta}Z^au_{\eta}|^2 \bigg) d\xi + C [ \mathcal {E}_1(u(t_0))+\mathcal {E}_2(u(t_0)) ] \nonumber \\
&\quad \qquad + C \sup_{0\leq t \leq t_0} E^{1/2}(u(t)) \cdot \mathcal {E}(u(t_0)) \text{.}
\end{align}

Using \eqref{cfffom88991}, \eqref{eexuyao88ejccjjjert55e}, \eqref{jkui89ghk}, and \eqref{duix998709}, we bound
\begin{align}\label{cfffom899ff98991}
&\sum_{\substack{|a|=2\\a_2\neq 0}}\int_{\Sigma^{-}_{t_0,\xi_0}}|\langle \eta\rangle^{1+\delta} Z^au_{\eta}|^2dx+\sum_{\substack{|a|=2\\a_2\neq 0}}\int_{C^{-}_{t_0,\xi_0}}|\langle \eta\rangle^{1+\delta} Z^au_{\eta}|^2 \nonumber \\
&\quad \leq C\varepsilon^2+CA^3\varepsilon^3 + C \sum_{\substack{|a|=2\\a_2\neq 0}}\int_{C^{-}_{t_0,\xi_0}}|\widetilde{e}_2(t,x)| \nonumber \\
&\quad \qquad + C \sum_{\substack{|a|=2\\a_2\neq 0}}\int_{-\infty}^{\xi_0}\langle \xi\rangle^{-1-\delta} \bigg( \int_{C^{-}_{t_0,\xi}}|\langle \eta\rangle^{1+\delta}Z^au_{\eta}|^2 \bigg) d\xi \text{.}
\end{align}
Applying the Gronwall inequality from Lemma \ref{gronwall}, along with \eqref{key2}, yields that
\begin{align}\label{cfffff98991}
&\sum_{\substack{|a|=2\\a_2\neq 0}}\int_{\Sigma^{-}_{t_0,\xi_0}}|\langle \eta\rangle^{1+\delta} Z^au_{\eta}|^2dx+\sum_{\substack{|a|=2\\a_2\neq 0}}\int_{C^{-}_{t_0,\xi_0}}|\langle \eta\rangle^{1+\delta} Z^au_{\eta}|^2 \nonumber \\
&\quad \leq C\varepsilon^2 + CA^3\varepsilon^3 + C\sum_{\substack{|a|=2\\a_2\neq 0}}\int_{C^{-}_{t_0,\xi_0}}|\widetilde{e}_2(t,x)|+C\sum_{\substack{|a|=2\\a_2\neq 0}}\int_{-\infty}^{\xi_0}\langle \xi\rangle^{-1-\delta} \bigg( \int_{C^{-}_{t_0,\xi}}|\widetilde{e}_2(t,x)| \bigg) d\xi \nonumber \\
&\quad \leq C\varepsilon^2+CA^3\varepsilon^3+C\sum_{\substack{|a|=2\\a_2\neq 0}}\|\langle \xi\rangle^{-1-\delta}\widetilde{e}_2\|_{L^1_{t,x}(S_{t_0})} + C\sum_{\substack{|a|=2\\a_2\neq 0}}\int_{C^{-}_{t_0,\xi_0}}|\widetilde{e}_2(t,x)| \text{.}
\end{align}
Multiplying \eqref{cfffff98991} by $\langle \xi_0\rangle^{-1-\delta}$ and then integrating respect to $\xi_0$ over $\mathbb{R}$ yields
\begin{align}\label{cfffff98991yyyy}
&\sum_{\substack{|a|=2\\a_2\neq 0}}\int_{\mathbb{R}}\langle \xi_0\rangle^{-1-\delta} \bigg( \int_{C^{-}_{t_0,\xi_0}}|\langle \eta\rangle^{1+\delta} Z^au_{\eta}|^2 \bigg) d\xi_0 \nonumber \\
&\quad \leq C\varepsilon^2+CA^3\varepsilon^3+C\sum_{\substack{|a|=2\\a_2\neq 0}}\|\langle \xi\rangle^{-1-\delta}\widetilde{e}_2\|_{L^1_{t,x}(S_{t_0})} \nonumber \\
&\quad \qquad + C\sum_{\substack{|a|=2\\a_2\neq 0}}\int_{\mathbb{R}}\langle \xi_0\rangle^{-1-\delta} \bigg( \int_{C^{-}_{t_0,\xi_0}}|\widetilde{e}_2(t,x)| \bigg) d\xi_0 \text{.}
\end{align}

From \eqref{key2} and \eqref{cfffff98991yyyy}, we obtain
\begin{align}\label{cfffff9kk8991}
\overline{\mathcal{E}}_3 (u(t_0)) &= \sum_{\substack{|a|=2\\a_2\neq 0}}\|\langle \xi\rangle^{-\frac{1+\delta}{2}}\langle \eta\rangle^{1+\delta} Z^au_{\eta}\|_{L^2_{t,x}(S_{t_0})}^2 \nonumber \\
&\leq C\varepsilon^2+CA^3\varepsilon^3 + C\sum_{\substack{|a|=2\\a_2\neq 0}}\|\langle \xi\rangle^{-1-\delta}\widetilde{e}_2\|_{L^1_{t,x}(S_{t_0})} \text{.}
\end{align}
By \eqref{eeejccjjjert55e999} and Lemma \ref{xu899quasilinear }, we have
\begin{align}\label{eeejccjjjddddddert55e}
&\|\langle \xi\rangle^{-1-\delta}\widetilde{e}_2(t,x)\|_{L^1_{t,x}(S_{t_0})} \nonumber \\
&\quad \leq C\|\langle \xi\rangle^{-\frac{1+\delta}{2}}\langle \eta\rangle^{1+\delta}u_{\eta}\|_{L^2_{t}L^{\infty}_{x}(S_{t_0})}\|\langle \xi\rangle^{-\frac{1+\delta}{2}}\langle \eta\rangle^{1+\delta}Z^au_{\eta}\|_{L^2_{t,x}(S_{t_0})}\|\langle \xi\rangle^{1+\delta}Z^au_{\xi}\|_{L^{\infty}_{t}L^{2}_{x}(S_{t_0})} \nonumber \\
&\quad \leq C \sup_{0 \leq t \leq t_0} E^{1/2}(u(t)) \cdot \mathcal{E} (u(t_0)) \text{.}
\end{align}
Combining \eqref{cfffff9kk8991} and \eqref{eeejccjjjddddddert55e} gives
\begin{align}\label{xui8962227yh}
\overline{\mathcal{E}}_3(u(t_0)) \leq C\varepsilon^2+CA^3\varepsilon^3 \text{.}
\end{align}

It remains to estimate $\overline{{E}}_3(u(t_0))$.
For this, we integrate \eqref{highgh7811} over $S_{t_0}$:
\begin{align}\label{com8uuu89919900}
\int_{\Sigma_{t_0}} e_2(t,x)dx &= \int_{\Sigma_{0}} e_2(0,x) dx + \int_{\Sigma_{t_0}}\overline{\overline{{e}}}_2(t,x)dx-\int_{\Sigma_{0}}\overline{\overline{{e}}}_2(0,x) dx \nonumber \\
&\qquad + \int_{S_{t_0}}q_2(t,x) +2\int_{S_{t_0}}\phi(\eta)Z^au^{T}_{\eta}H_a \text{.}
\end{align}
It follows from \eqref{com8uuu89919900} and \eqref{cjjjjom2} that
\begin{align}\label{cfffoyinguangm88991}
\overline{{E}}_3(u(t_0)) &= \sum_{\substack{|a|=2\\a_2\neq 0}}\|\langle \eta\rangle^{1+\delta} Z^au_{\eta}\|_{L^2_{x}(\Sigma_{t_0})} \nonumber \\
&\leq C \varepsilon^2+\sum_{\substack{|a|=2\\a_2\neq 0}} \|\overline{\overline{{e}}}_2\|_{L^1_{x}(\Sigma_{t_0})} + C\sum_{\substack{|a|=2\\a_2\neq 0}}\|q_2\|_{L^1_{t,x}(S_{t_0})} \nonumber \\
&\qquad + C\sum_{\substack{|a|=2\\a_2\neq 0}}\|\phi(\eta)Z^au^{T}_{\eta}H_a\|_{L^1_{t,x}(S_{t_0})} \text{.}
\end{align}
From \eqref{eeejffffccjjjert55e}, \eqref{esuiyujjanjjert55e}, and \eqref{xui8962227yh}, we then obtain
\begin{align}\label{jkui89ghkkiik}
&\sum_{\substack{|a|=2\\a_2\neq 0}}\|q_2\|_{L^1_{t,x}(S_{t_0})} + \sum_{\substack{|a|=2\\a_2\neq 0}}\|\phi(\eta)Z^au^{T}_{\eta}H_a\|_{L^1_{t,x}(S_{t_0})} \nonumber \\
&\quad \leq C \sum_{\substack{|a|=2\\a_2\neq 0}} \|\langle \xi\rangle^{-\frac{1+\delta}{2}}\langle \eta\rangle^{1+\delta}Z^bu_{\eta}\|_{L_{t,x}^2({S_{t_0}})}^2+C \sum_{|b|\leq 1} \|\langle \xi\rangle^{-\frac{1+\delta}{2}}\langle \eta\rangle^{1+\delta}Z^bu_{\eta}\|_{L_{t,x}^2({S_{t_0}})}^2 \nonumber \\
&\quad \qquad + C\sum_{|a|\leq 2}\|\langle\xi\rangle^{-\frac{1+\delta}{2}}\langle \eta\rangle^{1+\delta}Z^{a}u_{\eta}\|_{L_{t,x}^2({S_{t_0}})}\sum_{|b|\leq 1}\|\langle\xi\rangle^{-\frac{1+\delta}{2}}\langle \eta\rangle^{1+\delta}Z^{b}u_{\eta}\|_{L_{t}^2L_{x}^{\infty}({S_{t_0}})} \nonumber \\
&\quad \qquad \qquad \cdot \sum_{|c|\leq 2}\|\langle\xi\rangle^{1+\delta}Z^{c}u_{\xi}\|_{L_{t}^{\infty}L_{x}^{2}({S_{t_0}})} \nonumber \\
&\quad \qquad + C \sum_{|b|\leq 2}\|\langle\xi\rangle^{-\frac{1+\delta}{2}}\langle \eta\rangle^{1+\delta}Z^{b}u_{\eta}\|_{L_{t,x}^2({S_{t_0}})}^2\sum_{|c|\leq 1}\|\langle\xi\rangle^{1+\delta}Z^{c}u_{\xi}\|_{L_{t,x}^{\infty}({S_{t_0}})} \nonumber \\
&\quad \leq C \overline{\mathcal{E}}_3(u(t_0)) +C [ \mathcal {E}_1(u(t_0))+ \mathcal {E}_2(u(t_0)) ] + C \sup_{0 \leq t \leq t_0} E^{1/2}(u(t)) \cdot \mathcal{E}(u(t_0)) \nonumber \\
&\quad \leq C\varepsilon^2+CA^3\varepsilon^3 \text{.}
\end{align}

Finally, from \eqref{cfffoyinguangm88991}, \eqref{eexuyao88ejccjjjert55e}, and \eqref{jkui89ghkkiik}, we conclude
\begin{align}\label{Pthcom234}
\overline{E}_3(u(t_0))\leq C\varepsilon^2+CA^3\varepsilon^3 \text{.}
\end{align}
Combining \eqref{xui8962227yh} and \eqref{Pthcom234} then yields the estimate
\begin{align}\label{Pthcom234888999}
\overline{E}_3(u(t_0)) + \overline{\mathcal{E}}_3(u(t_0))\leq C\varepsilon^2+CA^3\varepsilon^3 \text{.}
\end{align}

\subsubsection{Estimates of $\widetilde{E}_3$ and $\widetilde{\mathcal{E}}_3$}

We now turn to $\widetilde{E}_3 (u(t_0))$ and $\widetilde{\mathcal{E}}_3 (u(t_0))$, which can not be treated by the usual integration by parts argument.
Similarly to \cite{MR4098041}, we will instead use the system \eqref{noytingthat333} directly, while noting that $\overline{{E}}_3 (u(t_0))$ and $\overline{{{\mathcal{E}}}}_3 (u(t_0))$ have already been controlled.

First, we differentiate \eqref{noytingthat333} and write
\begin{align}\label{labke888dee99}
\big( I - \widetilde{A}_1 \big) u_{\xi\xi\eta}=G_3 \text{,}
\end{align}
where
\begin{align}
G_3 &= \widetilde{A}_2u_{\xi\eta\eta}+\widetilde{A}_3u_{\xi\xi\xi} + \partial_{\xi}\widetilde{A}_1u_{\xi\eta}+\partial_{\xi}\widetilde{A}_2u_{\eta\eta}+\partial_{\xi}\widetilde{A}_3u_{\xi\xi} \nonumber \\
&\qquad + \widetilde{A}_3f^{(3)}(\xi)+\partial_{\xi}\widetilde{A}_3f''(\xi) + \partial_{\xi}\widetilde{F} \text{.}
\end{align}
It then follows from Lemma \ref{le55mmmad2222} that
\begin{align}\label{byui998yi7}
|G_3| &\leq C|f'(\xi)||u_{\xi\eta\eta}|+\big(|f'(\xi)|+|f''(\xi)|+|f^{(3)}(\xi)|\big)\sum_{|b|\leq 1}|Z^{b}u_{\eta}| \nonumber \\
&\qquad + C\sum_{|b|\leq 2}|Z^{b}u_{\eta}|\sum_{|c|\leq 1}|Z^{c}u_{\xi}|+C\sum_{|b|\leq 1}|Z^{b}u_{\eta}|\sum_{|c|\leq 2}|Z^{c}u_{\xi}| \nonumber \\
&\leq C|u_{\xi\eta\eta}|+C\sum_{|b|\leq 1}|Z^{b}u_{\eta}|+C\sum_{|b|\leq 2}|Z^{b}u_{\eta}|\sum_{|c|\leq 1}|Z^{c}u_{\xi}| \nonumber \\
&\qquad + C\sum_{|b|\leq 1}|Z^{b}u_{\eta}|\sum_{|c|\leq 2}|Z^{c}u_{\xi}| \text{.}
\end{align}

Multiplying both sides of \eqref{labke888dee99} by $u^{T}_{\xi\xi\eta}$, we have
\begin{align}\label{key8990}
u^{T}_{\xi\xi\eta}\big(I-\widetilde{A}_1\big)u_{\xi\xi\eta}=u^{T}_{\xi\xi\eta}G_3 \text{.}
\end{align}
By \eqref{hyper10}, we have the lower bound
\begin{align}\label{fgy78855}
u^{T}_{\xi\xi\eta}\big(I-\widetilde{A}_1\big)u_{\xi\xi\eta} \geq \frac{\lambda}{2}|u_{\xi\xi\eta}|^2 \text{.}
\end{align}
Thus, by \eqref{key8990} and \eqref{fgy78855},
\begin{align}
\label{keygggg89hh90} |\langle \eta\rangle^{1+\delta} u_{\xi\xi\eta}| &\leq C\langle \eta\rangle^{1+\delta} |G_3| \text{,} \\
\label{similarl22} |\langle \xi\rangle^{-\frac{1+\delta}{2}}\langle \eta\rangle^{1+\delta} u_{\xi\xi\eta}| &\leq C\langle \xi\rangle^{-\frac{1+\delta}{2}}\langle \eta\rangle^{1+\delta} |G_3| \text{.}
\end{align}

By \eqref{byui998yi7}, \eqref{keygggg89hh90}, and Lemma \ref{xu899quasilinear }, we have
\begin{align}\label{keygggddg89ddhh90}
\|\langle \eta\rangle^{1+\delta} u_{\xi\xi\eta}\|_{L^2_{x}(\Sigma_{t_0})} &\leq C\|\langle \eta\rangle^{1+\delta} G_3\|_{L^2_{x}(\Sigma_{t_0})} \nonumber \\
&\leq C\|\langle \eta\rangle^{1+\delta}u_{\xi\eta\eta}\|_{L^2_{x}(\Sigma_{t_0})}+C\sum_{|b|\leq 1}\|\langle \eta\rangle^{1+\delta}Z^{b}u_{\eta}\|_{L^2_{x}(\Sigma_{t_0})} \nonumber \\
&\qquad + C\sum_{|b|\leq 2}\|\langle \eta\rangle^{1+\delta}Z^{b}u_{\eta}\|_{L^2_{x}(\Sigma_{t_0})}\sum_{|c|\leq 1}\|Z^{c}u_{\xi}\|_{L^{\infty}_{x}(\Sigma_{t_0})} \nonumber \\
&\qquad + C\sum_{|b|\leq 1}\|\langle \eta\rangle^{1+\delta}Z^{b}u_{\eta}\|_{L^{\infty}_{x}(\Sigma_{t_0})}\sum_{|c|\leq 2}\|Z^{c}u_{\xi}\|_{L^{2}_{x}(\Sigma_{t_0})} \nonumber \\
&\leq C\|\langle \eta\rangle^{1+\delta}u_{\xi\eta\eta}\|_{L^2_{x}(\Sigma_{t_0})} + C \big[ E^{1/2}_1(u(t_0))+E^{1/2}_2(u(t_0)) \big] \nonumber \\
&\qquad + C E (u(t_0)) \text{.}
\end{align}
Similarly, \eqref{byui998yi7}, \eqref{similarl22}, and Lemma \ref{xu899quasilinear } imply
\begin{align}\label{keygggfgghh90}
&\|\langle \xi\rangle^{-\frac{1+\delta}{2}}\langle \eta\rangle^{1+\delta} u_{\xi\xi\eta}\|_{L^2_{x}(\Sigma_{t_0})} \nonumber \\
&\quad \leq C\|\langle \xi\rangle^{-\frac{1+\delta}{2}}\langle \eta\rangle^{1+\delta} G_3\|_{L^2_{t,x}(S_{t_0})} \nonumber \\
&\quad \leq C \|\langle \xi\rangle^{-\frac{1+\delta}{2}}\langle \eta\rangle^{1+\delta}u_{\xi\eta\eta}\|_{L^2_{t,x}(S_{t_0})}+C\sum_{|b|\leq 1}\|\langle \xi\rangle^{-\frac{1+\delta}{2}}\langle \eta\rangle^{1+\delta}Z^{b}u_{\eta}\|_{L^2_{t,x}(S_{t_0})} \nonumber \\
&\quad \qquad + C\sum_{|b|\leq 2}\|\langle \xi\rangle^{-\frac{1+\delta}{2}}\langle \eta\rangle^{1+\delta}Z^{b}u_{\eta}\|_{L^2_{t,x}(S_{t_0})}\sum_{|c|\leq 1}\|Z^{c}u_{\xi}\|_{L^{\infty}_{t,x}(S_{t_0})} \nonumber \\
&\quad \qquad + C\sum_{|b|\leq 1}\|\langle \xi\rangle^{-\frac{1+\delta}{2}}\langle \eta\rangle^{1+\delta}Z^{b}u_{\eta}\|_{L^2_{t}L^{\infty}_{x}(S_{t_0})}\sum_{|c|\leq 2}\|Z^{c}u_{\xi}\|_{L^{\infty}_{t}L^{2}_{x}(S_{t_0})} \nonumber \\
&\quad \leq C\|\langle \xi\rangle^{-\frac{1+\delta}{2}}\langle \eta\rangle^{1+\delta}u_{\xi\eta\eta}\|_{L^2_{t,x}(S_{t_0})} + C \big[ \mathcal{E}^{1/2}_1(u(t_0))+\mathcal {E}^{1/2}_2(u(t_0)) \big] \nonumber \\
&\quad \qquad + C \sup_{0 \leq t \leq t_0} E^{1/2}(u(t)) \cdot \mathcal{E}^{1/2}(u(t_0)) \text{.}
\end{align}

Next, observe that \eqref{noytingthat333} implies
\begin{align}\label{tulta888}
\big(I-\widetilde{A}_1\big)u_{\eta\eta\xi}=\widetilde{G}_3 \text{,}
\end{align}
where
\begin{align}
\widetilde{G}_3 &= \widetilde{A}_2u_{\eta\eta\eta}+\widetilde{A}_3u_{\xi\eta\xi}+
\partial_{\eta}\widetilde{A}_1u_{\eta\xi}+\partial_{\eta}\widetilde{A}_2u_{\eta\eta}+\partial_{\eta}\widetilde{A}_3u_{\xi\xi} +\partial_{\eta}\widetilde{A}_3f''(\xi) +\partial_{\eta}\widetilde{F} \text{.}
\end{align}
From Lemma \ref{le55mmmad2222} and \eqref{fgty788899}, we have
\begin{align}\label{byui9f98yi7ff8999}
|\widetilde{G}_3| &\leq C|f'(\xi)||u_{\eta\eta\eta}| + C \big(|f'(\xi)|+|f''(\xi)|\big)\sum_{|b|\leq 1}|Z^{b}u_{\eta}| \nonumber \\
&\qquad + C\sum_{|b|\leq 2}|Z^{b}u_{\eta}|\sum_{|c|\leq 1}|Z^{c}u_{\xi}|+C\sum_{|b|\leq 1}|Z^{b}u_{\eta}|\sum_{|c|\leq 2}|Z^{c}u_{\xi}| \nonumber \\
&\leq C\langle \xi\rangle^{-\frac{3}{2}(1+\delta)}|u_{\eta\eta\eta}|+C\langle \xi\rangle^{-\frac{3}{2}(1+\delta)}\sum_{|b|\leq 1}|Z^{b}u_{\eta}| \nonumber \\
&\qquad + C\sum_{|b|\leq 2}|Z^{b}u_{\eta}|\sum_{|c|\leq 1}|Z^{c}u_{\xi}|+C\sum_{|b|\leq 1}|Z^{b}u_{\eta}|\sum_{|c|\leq 2}|Z^{c}u_{\xi}| \text{.}
\end{align}
Similarly to \eqref{keygggg89hh90} and \eqref{similarl22}, we obtain from \eqref{hyper10} and \eqref{tulta888} that
\begin{align}
\label{keygggg89hh98880} |\langle \xi\rangle^{1+\delta}u_{\eta\eta\xi}| &\leq C\langle\xi\rangle^{1+\delta} |\widetilde{G}_3| \text{,} \\
\label{simil888arl22} |\langle \eta\rangle^{-\frac{1+\delta}{2}}\langle \xi\rangle^{1+\delta} u_{\eta\eta\xi}| &\leq C\langle\eta\rangle^{-\frac{1+\delta}{2}}\langle \xi\rangle^{1+\delta} |\widetilde{G}_3| \text{.}
\end{align}

By \eqref{keygggg89hh98880}, \eqref{byui9f98yi7ff8999}, and Lemma \ref{xu899quasilinear }, we bound
\begin{align}\label{xuii90878}
\|\langle \xi\rangle^{1+\delta}u_{\eta\eta\xi}\|_{L^2_{x}(\Sigma_{t_0})} &\leq C\|\langle \xi\rangle^{1+\delta} \widetilde{G}_3\|_{L^2_{x}(\Sigma_{t_0})} \nonumber \\
&\leq C\|u_{\eta\eta\eta}\|_{L^2_{x}(\Sigma_{t_0})}+C\sum_{|b|\leq 1}\|Z^{b}u_{\eta}\|_{L^2_{x}(\Sigma_{t_0})} \nonumber \\
&\qquad + C\sum_{|b|\leq 2}\|Z^{b}u_{\eta}\|_{L^{2}_{x}(\Sigma_{t_0})}\sum_{|c|\leq 1}\|\langle \xi\rangle^{1+\delta}Z^{c}u_{\xi}\|_{L^{\infty}_{x}(\Sigma_{t_0})} \nonumber \\
&\qquad +C\sum_{|b|\leq 1}\|Z^{b}u_{\eta}\|_{L^{\infty}_{x}(\Sigma_{t_0})}\sum_{|c|\leq 2}\|\langle \xi\rangle^{1+\delta}Z^{c}u_{\xi}\|_{L^2_{x}(\Sigma_{t_0})} \nonumber \\
&\leq C\|\langle \eta\rangle^{1+\delta}u_{\eta\eta\eta}\|_{L^2_{x}(\Sigma_{t_0})} + C \big[ E^{1/2}_1(u(t_0))+E^{1/2}_2(u(t_0)) \big] \nonumber \\
&\qquad + CE(u(t_0)) \text{.}
\end{align}
Likewise, \eqref{simil888arl22}, \eqref{byui9f98yi7ff8999}, and Lemma \ref{xu899quasilinear } imply
\begin{align}\label{xuyi8965tu}
&\|\langle \eta\rangle^{-\frac{1+\delta}{2}}\langle \xi\rangle^{1+\delta}u_{\eta\eta\xi}\|_{L^2_{t,x}(S_{t_0})} \nonumber \\
&\quad \leq C\|\langle \eta\rangle^{-\frac{1+\delta}{2}}\langle \xi\rangle^{1+\delta} \widetilde{G}_3\|_{L^2_{t,x}(S_{t_0})} \nonumber \\
&\quad \leq C\|\langle \xi\rangle^{-\frac{1+\delta}{2}}u_{\eta\eta\eta}\|_{L^2_{t,x}(S_{t_0})}+C\sum_{|b|\leq 1}\|\langle \xi\rangle^{-\frac{1+\delta}{2}}Z^{b}u_{\eta}\|_{L^2_{t,x}(S_{t_0})} \nonumber \\
&\quad \qquad + C\sum_{|b|\leq 2}\|Z^{b}u_{\eta}\|_{L^{\infty}_{t}L^{2}_{x}(S_{t_0})}\sum_{|c|\leq 1}\|\langle \eta\rangle^{-\frac{1+\delta}{2}}\langle \xi\rangle^{1+\delta}Z^{c}u_{\xi}\|_{L^2_{t}L^{\infty}_{x}(S_{t_0})} \nonumber \\
&\quad \qquad + C\sum_{|b|\leq 1}\|Z^{b}u_{\eta}\|_{L^{\infty}_{t,x}(S_{t_0})}\sum_{|c|\leq 2}\|\langle \eta\rangle^{-\frac{1+\delta}{2}}\langle \xi\rangle^{1+\delta}Z^{c}u_{\xi}\|_{L^2_{t,x}(S_{t_0})} \nonumber \\
&\quad \leq C\|\langle \xi\rangle^{-\frac{1+\delta}{2}}\langle \eta\rangle^{1+\delta}u_{\eta\eta\eta}\|_{L^2_{t,x}(S_{t_0})} + C \big[ \mathcal{E}^{1/2}_1(u(t_0))+\mathcal {E}^{1/2}_2(u(t_0)) \big] \nonumber \\
&\quad \qquad + C \sup_{0\leq t\leq t_0} E^{1/2}(u(t)) \cdot \mathcal{E}^{1/2}(u(t_0)) \text{.}
\end{align}

Finally, putting together the bounds \eqref{keygggddg89ddhh90}, \eqref{xuii90878}, and \eqref{duix998709} yields
\begin{align}\label{xui9967eee444r88}
\widetilde{E}_3(u(t_0)) &\leq C\overline{{E}}_3(u(t_0)) + C \big[ E_1(u(t_0))+E_2(u(t_0)) \big] + CE^2(u(t_0)) \nonumber \\
&\leq C\overline{{E}}_3(u(t_0))+C\varepsilon^2+CA^3\varepsilon^3 \text{.}
\end{align}
Similarly, by \eqref{keygggfgghh90}, \eqref{xuyi8965tu}, and \eqref{duix998709},
\begin{align}\label{xui9967eee44erfrrrrrr4r88}
\widetilde{\mathcal {E}}_3(u(t_0)) &\leq C\overline{{\mathcal {E}}}_3(u(t_0))+C\big(\mathcal {E}_1(u(t_0))+\mathcal {E}_2(u(t_0)) \big)+C\sup_{0\leq t\leq t_0}E(u(t))\mathcal {E}(u(t_0)) \nonumber \\
&\leq C\overline{{\mathcal {E}}}_3(u(t_0))+C\varepsilon^2+CA^3\varepsilon^3 \text{.}
\end{align}

\subsubsection{Estimates of $\widehat{{E}}_3$ and $\widehat{\mathcal{E}}_3$}

Finally, we estimate $\widehat{{E}}_3(u(t_0))$ and $\widehat{\mathcal{E}}_3 (u(t_0))$.
Again, to control the linear part, we make use of bounds we have already derived for the other top-order energies.

For any multi-index $a = (a_1,a_2)$, with $|a|=2$ and $a_1\neq 0$, the Leibniz rule gives
\begin{align}\label{quasiwavegao7jjjjjjj888}
Z^{a}u_{\xi\eta} &= \widetilde{A}_1Z^{a}u_{\xi\eta} + \widetilde{A}_2Z^{a}u_{\eta\eta}+\widetilde{A}_3Z^{a}u_{\xi\xi}+\widetilde{H}_a \text{,}
\end{align}
where
\begin{align}
\widetilde{H}_a &= \sum_{\substack{b+c=a\\c\neq a}}\lambda_{bc}\big(Z^{b}\widetilde{A}_1Z^{c}u_{\xi\eta}+Z^{b}\widetilde{A}_2Z^{c}u_{\eta\eta}+Z^{b}\widetilde{A}_3Z^{c}u_{\xi\xi}\big) \nonumber \\
&\qquad + \sum_{{b+c=a}}\lambda_{bc}Z^{b}\widetilde{A}_3Z^{c}f''(\xi) +Z^{a}\widetilde{F} \text{,}
\end{align}
where $\lambda_{bc}$ are integer constants.
Moreover, by Lemmas \ref{le55mmmad2222} and \ref{xhjk999k99},
\begin{align}\label{xuyao9jllzongyaoj566}
|\widetilde{H}_a| &\leq C\big(|f'(\xi)|+|f''(\xi)|+|f^{(3)}(\xi)|+|f^{(4)}(\xi)|\big)(\sum_{|b|=2}|Z^{b}u_{\eta}|+\sum_{|b|\leq 1}|Z^{b}u_{\eta}|) \nonumber \\
&\qquad + C\sum_{|b|\leq 1}|Z^{b}u_{\eta}|\sum_{|c|\leq 2}|Z^{c}u_{\xi}|+C\sum_{|b|\leq 2}|Z^{b}u_{\eta}|\sum_{|c|\leq 1}|Z^{c}u_{\xi}| \text{.}
\end{align}

Multiplying \eqref{quasiwavegao7jjjjjjj888} by $2\psi(\eta)\phi(\xi)Z^{a}u^{T}_{\xi}$ and noting the symmetry of $\widetilde{A}_1$, $\widetilde{A}_2$, and $\widetilde{A}_3$ yields
\begin{align}\label{lemkejjjy17888}
&\big(\psi(\eta)\phi(\xi)|Z^{a}u_{\xi}|^2\big)_{\eta}-\psi'(\eta)\phi(\xi)|Z^{a}u_{\xi}|^2 \nonumber \\
&\quad = \big(\psi(\eta)\phi(\xi)Z^{a}u_{\xi}^{{T}}\widetilde{A}_1Z^{a}u_{\xi}\big)_{\eta}-\psi'(\eta)\phi(\xi)Z^{a}u_{\xi}^{{T}}\widetilde{A}_1Z^{a}u_{\xi}
-\psi(\eta)\phi(\xi)Z^{a}u_{\xi}^{{T}}\partial_{\eta}\widetilde{A}_1Z^{a}u_{\xi} \nonumber \\
&\quad \qquad + \big(2\psi(\eta)\phi(\xi)Z^{a}u_{\xi}^{{T}}\widetilde{A}_2Z^{a}u_{\eta}\big)_{\eta}-2\psi'(\eta)\phi(\xi)Z^{a}u_{\xi}^{{T}}\widetilde{A}_2Z^{a}u_{\eta} \nonumber \\
&\quad \qquad - 2\psi(\eta)\phi(\xi)Z^{a}u_{\xi}^{{T}}\partial_{\eta}\widetilde{A}_2Z^{a}u_{\eta} - \big(\psi(\eta)\phi(\xi)Z^{a}u_{\eta}^{{T}}\widetilde{A}_2Z^{a}u_{\eta}\big)_{\xi} \nonumber \\
&\quad \qquad + \psi(\eta)\phi'(\xi)Z^{a}u_{\eta}^{{T}}\widetilde{A}_2Z^{a}u_{\eta} + \psi(\eta)\phi(\xi)Z^{a}u_{\eta}^{{T}}\partial_{\xi}\widetilde{A}_2Z^{a}u_{\eta} \nonumber \\
&\quad \qquad + \big(\psi(\eta)\phi(\xi)Z^{a}u_{\xi}^{{T}}\widetilde{A}_3Z^{a}u_{\xi}\big)_{\xi} - \psi(\eta)\phi'(\xi)Z^{a}u_{\xi}^{{T}}\widetilde{A}_3Z^{a}u_{\xi} \nonumber \\
&\quad \qquad - \psi(\eta)\phi(\xi)Z^{a}u_{\xi}^{{T}}\partial_{\xi}\widetilde{A}_3Z^{a}u_{\xi} + 2\psi(\eta)\phi(\xi)Z^{a}u_{\xi}^{{T}}\widetilde{H}_a \text{.}
\end{align}
Integrating both sides of \eqref{lemkejjjy17888} over $S_{t_0}$ results in the identity
\begin{align}\label{itfolojjjjj888wjk99}
\int_{\Sigma_{t_0}}e_3(t_0,x)dx+\int_{S_{t_0}}p_3(t,x) &=\int_{\Sigma_{0}}e_3(0,x)dx+\int_{\Sigma_{t_0}}\widetilde{e}_3(t_0,x)dx-\int_{\Sigma_{0}}\widetilde{e}_3(0,x)dx \nonumber \\
&\qquad + \int_{S_{t_0}}q_3(t,x)+2\int_{S_{t_0}}\psi(\eta)\phi(\xi)Z^{a}u_{\xi}^{{T}}\widetilde{H}_a \text{,}
\end{align}
where
\begin{align}
e_3 &= \psi(\eta)\phi(\xi)Z^{a}u_{\xi}^{{T}}(I-\widetilde{A}_1)Z^{a}u_{\xi} \text{,} \\
p_3 &= -\psi'(\eta)\phi(\xi)Z^{a}u_{\xi}^{{T}}(I-\widetilde{A}_1)Z^{a}u_{\xi} \text{,} \\
\widetilde{e}_3&=2\psi(\eta)\phi(\xi)Z^{a}u_{\xi}^{{T}}\widetilde{A}_2Z^{a}u_{\eta}-\psi(\eta)\phi(\xi)Z^{a}u_{\eta}^{{T}}\widetilde{A}_2Z^{a}u_{\eta} \nonumber \\
&\qquad + \psi(\eta)\phi(\xi)Z^{a}u_{\xi}^{{T}}\widetilde{A}_3Z^{a}u_{\xi} \text{,} \\
q_3 &= -\psi(\eta)\phi(\xi)Z^{a}u_{\xi}^{{T}}\partial_{\eta}\widetilde{A}_1Z^{a}u_{\xi}-2\psi'(\eta)\phi(\xi)Z^{a}u_{\xi}^{{T}}\widetilde{A}_2Z^{a}u_{\eta} \nonumber \\
&\qquad - 2\psi(\eta)\phi(\xi)Z^{a}u_{\xi}^{{T}}\partial_{\eta}\widetilde{A}_2Z^{a}u_{\eta}+\psi(\eta)\phi'(\xi)Z^{a}u_{\eta}^{{T}}\widetilde{A}_2Z^{a}u_{\eta} \nonumber \\
&\qquad + \psi(\eta)\phi(\xi)Z^{a}u_{\eta}^{{T}}\partial_{\xi}\widetilde{A}_2Z^{a}u_{\eta} - \psi(\eta)\phi'(\xi)Z^{a}u_{\xi}^{{T}}\widetilde{A}_3Z^{a}u_{\xi} \nonumber \\
&\qquad - \psi(\eta)\phi(\xi)Z^{a}u_{\xi}^{{T}}\partial_{\xi}\widetilde{A}_3Z^{a}u_{\xi} \text{.}
\end{align}

In view of \eqref{hyper10}, \eqref{weight1}, and \eqref{weight3}, we have that
\begin{align}\label{cofffmm2}
e_3(t,x)\geq \frac{\lambda}{2c}|\langle \xi\rangle^{1+\delta} Z^au_{\xi}|^2 \text{,} \qquad p_3(t,x)\geq \frac{\lambda}{2c}|\langle \eta\rangle^{-\frac{1+\delta}{2}} \langle \xi\rangle^{1+\delta} Z^au_{\xi}|^2 \text{.}
\end{align}
Moreover, by Lemma \ref{le55mmmad2222} and \eqref{fgty788899},
\begin{align}\label{eeejccjjjertjuuu55e}
|\widetilde{e}_3(t,x)| &\leq C|f'(\xi)|\langle \xi\rangle^{2+2\delta}|Z^au_{\xi}||Z^au_{\eta}|+C|f'(\xi)|\langle \xi\rangle^{2+2\delta}|Z^au_{\eta}|^2 \nonumber \\
&\qquad + C\langle \xi\rangle^{2+2\delta}|u_{\xi}||Z^au_{\xi}||Z^au_{\eta}|+C\langle \xi\rangle^{2+2\delta}|u_{\xi}||Z^au_{\eta}|^2 + C\langle \xi\rangle^{2+2\delta}|u_{\eta}||Z^au_{\xi}|^2 \nonumber \\
&\leq C|\langle \xi\rangle^{1+\delta}Z^au_{\xi}||Z^au_{\eta}|+C|Z^au_{\eta}|^2 \nonumber \\
&\qquad + C|\langle \xi\rangle^{1+\delta}u_{\xi}|\sum_{|b|=2}|\langle \xi\rangle^{1+\delta}Z^bu_{\xi}||Z^au_{\eta}| + C|u_{\eta}||\langle \xi\rangle^{1+\delta}Z^au_{\xi}|^2 \text{.}
\end{align}
Here, we also noted that since $a_1\neq 0$,
\begin{align}\label{fact2}
|Z^{a}u_{\eta}| = |\partial_{\xi}^{a_1}\partial_{\eta}^{a_2}u_{\eta}|\leq \sum_{|b|=2}|Z^{b}u_{\xi}| \text{.}
\end{align}

By Lemma \ref{le55mmmad2222}, \eqref{fgty788899}, and \eqref{fact2}, we estimate
\begin{align}\label{q23333}
|q_3(t,x)| &\leq C|f'(\xi)|\langle \xi\rangle^{2+2\delta}|Z^au_{\xi}||Z^au_{\eta}|+C(|f'(\xi)|+|f''(\xi)|)\langle \xi\rangle^{2+2\delta}|Z^au_{\eta}|^2 \nonumber \\
&\qquad + C|\langle \xi\rangle^{1+\delta}Z^{a}u_{\xi}|^2(|u_{\eta}|+|u_{\xi\eta}|+|u_{\eta\eta}|) \nonumber \\
&\qquad + C\langle \xi\rangle^{2+2\delta}|Z^{a}u_{\xi}||Z^{a}u_{\eta}|(|u_{\xi}|+|u_{\xi\eta}|)+C\langle \xi\rangle^{2+2\delta}|Z^{a}u_{\eta}|^2(|u_{\xi}|+|u_{\xi\xi}|+|u_{\xi\eta}|) \nonumber \\
&\leq C|\langle\eta\rangle^{-\frac{1+\delta}{2}}\langle \xi\rangle^{1+\delta}Z^au_{\xi}||\langle\xi\rangle^{-\frac{1+\delta}{2}}\langle \eta\rangle^{1+\delta}Z^au_{\eta}| + C|\langle\xi\rangle^{-\frac{1+\delta}{2}}\langle \eta\rangle^{1+\delta}Z^au_{\eta}|^2 \nonumber \\
&\qquad + C|\langle\eta\rangle^{-\frac{1+\delta}{2}}\langle \xi\rangle^{1+\delta}Z^au_{\xi}|^2\sum_{|b|\leq 1}|\langle\eta\rangle^{1+\delta}Z^bu_{\eta}| \nonumber \\
&\qquad + C|\langle \eta\rangle^{1+\delta}Z^au_{\eta}|\sum_{|b|= 2}|\langle\eta\rangle^{-\frac{1+\delta}{2}}\langle \xi\rangle^{1+\delta}Z^bu_{\xi}|\sum_{|c|\leq 1 }|\langle\eta\rangle^{-\frac{1+\delta}{2}}\langle \xi\rangle^{1+\delta}Z^cu_{\xi}| \text{.}
\end{align}
In addition, by \eqref{xuyao9jllzongyaoj566} and \eqref{fgty788899},
\begin{align}\label{xui8dddd9999}
&|\psi(\eta)\phi(\xi)Z^{a}u_{\xi}^{{T}}\widetilde{H}_a| \nonumber \\
&\quad \leq C|\langle\eta\rangle^{-\frac{1+\delta}{2}}\langle \xi\rangle^{1+\delta}Z^{a}u_{\xi}|\big(\sum_{|b|=2}|\langle\xi\rangle^{-\frac{1+\delta}{2}}\langle \eta\rangle^{1+\delta}Z^{b}u_{\eta}|+\sum_{|b|\leq 1}|\langle\xi\rangle^{-\frac{1+\delta}{2}}\langle \eta\rangle^{1+\delta}Z^{b}u_{\eta}|\big) \nonumber \\
&\quad \qquad + C|\langle\eta\rangle^{-\frac{1+\delta}{2}}\langle \xi\rangle^{1+\delta}Z^{a}u_{\xi}|\sum_{|b|\leq 1}|\langle \eta\rangle^{1+\delta}Z^{b}u_{\eta}|\sum_{|c|\leq 2}|\langle\eta\rangle^{-\frac{1+\delta}{2}}\langle \xi\rangle^{1+\delta}Z^{c}u_{\xi}| \nonumber \\
&\quad \qquad + C|\langle\eta\rangle^{-\frac{1+\delta}{2}}\langle \xi\rangle^{1+\delta}Z^{a}u_{\xi}|\sum_{|b|\leq 2}|\langle \eta\rangle^{1+\delta}Z^{b}u_{\eta}|\sum_{|c|\leq 1}|\langle\eta\rangle^{-\frac{1+\delta}{2}}\langle \xi\rangle^{1+\delta}Z^{c}u_{\xi}| \text{.}
 \end{align}

Now, from \eqref{itfolojjjjj888wjk99} and \eqref{cofffmm2}, we have that
\begin{align}\label{itfolojjj8899jj888wjk99}
\widehat{{E}}_3(u(t_0))+\widehat{{{\mathcal{E}}}}_3(u(t_0)) &= \sum_{\substack{|a|=2\\a_1\neq 0}}\|\langle \xi\rangle^{1+\delta} Z^au_{\xi}\|^2_{L^2_{x}(\Sigma_{t_0})} + \sum_{\substack{|a|=2\\a_1\neq 0}}\|\langle \eta\rangle^{-\frac{1+\delta}{2}}\langle \xi\rangle^{1+\delta} Z^au_{\xi}\|^2_{L^2_{t,x}(S_{t_0})} \nonumber \\
&\leq C\varepsilon^2+C\sum_{\substack{|a|=2\\a_1\neq 0}}\|\widetilde{e}_3\|_{L^1_{x}(\Sigma_{t_0})}+C\sum_{\substack{|a|=2\\a_1\neq 0}}\|q_3\|_{L^1_{t,x}(S_{t_0})} \nonumber \\
&\qquad + C\sum_{\substack{|a|=2\\a_1\neq 0}}\|\psi(\eta)\phi(\xi)Z^{a}u_{\xi}^{{T}}\widetilde{H}_a\|_{L^1_{t,x}(S_{t_0})} \text{.}
\end{align}
Moreover, by \eqref{eeejccjjjertjuuu55e},
\begin{align}\label{eeejccfgg55e}
\|\widetilde{e}_3\|_{L^1_{x}(\Sigma_{t_0})} &\leq \frac{1}{100}\|\langle \xi\rangle^{1+\delta}Z^au_{\xi}\|^2_{L^2_{x}(\Sigma_{t_0})}+C\|Z^au_{\eta}\|_{L^2_{x}(\Sigma_{t_0})}^2 \nonumber \\
&\qquad + C \|\langle \xi\rangle^{1+\delta}u_{\xi}\|_{L^{\infty}_{x}(\Sigma_{t_0})}\sum_{|b|=2}\|\langle \xi\rangle^{1+\delta}Z^bu_{\xi}\|_{L^2_{x}(\Sigma_{t_0})}\|Z^au_{\eta}\|_{L^2_{x}(\Sigma_{t_0})} \nonumber \\
&\qquad + C \|u_{\eta}\|_{L^{\infty}_{x}(\Sigma_{t_0})}\|\langle \xi\rangle^{1+\delta}Z^au_{\xi}\|_{L^2_{x}(\Sigma_{t_0})}^2 \nonumber \\
&\leq \frac{1}{100}\widehat{{E}}_3(u(t_0))+C\overline{E}_3(u(t_0))+C\widetilde{E}_3(u(t_0))+CE^{3/2}(u(t_0)) \text{,}
\end{align}
while by \eqref{q23333}, we have
\begin{align}\label{q23899333}
\|q_3\|_{L^1_{t,x}(S_{t_0})} &\leq \frac{1}{100}\|\langle\eta\rangle^{-\frac{1+\delta}{2}}\langle \xi\rangle^{1+\delta}Z^au_{\xi}\|^2_{L^2_{t,x}(S_{t_0})}
+C\|\langle\xi\rangle^{-\frac{1+\delta}{2}}\langle \eta\rangle^{1+\delta}Z^au_{\eta}\|^2_{L^2_{t,x}(S_{t_0})} \nonumber \\
&\qquad + C\|\langle\eta\rangle^{-\frac{1+\delta}{2}}\langle \xi\rangle^{1+\delta}Z^au_{\xi}\|_{L^2_{t,x}(S_{t_0})}^2\sum_{|b|\leq 1}\|\langle\eta\rangle^{1+\delta}Z^bu_{\eta}\|_{L^{\infty}_{t,x}(S_{t_0})} \nonumber \\
&\qquad + C\|\langle \eta\rangle^{1+\delta}Z^au_{\eta}\|_{L^{\infty}_{t}L^2_{x}(S_{t_0})}\sum_{|b|= 2}\|\langle\eta\rangle^{-\frac{1+\delta}{2}}\langle \xi\rangle^{1+\delta}Z^bu_{\xi}\|_{L^2_{t,x}(S_{t_0})} \nonumber \\
&\qquad \qquad \cdot \sum_{|c|\leq 1 }\|\langle\eta\rangle^{-\frac{1+\delta}{2}}\langle \xi\rangle^{1+\delta}Z^cu_{\xi}\|_{L^{2}_{t}L^{\infty}_{x}(S_{t_0})} \nonumber \\
&\leq \frac{1}{100}\widehat{\mathcal {E}}_3(u(t_0))+C\overline{\mathcal {E}}_3(u(t_0))+C\widetilde{\mathcal {E}}_3(u(t_0)) \nonumber \\
&\qquad + C \sup_{0 \leq t \leq t_0} E^{1/2} (u(t)) \cdot \mathcal{E}(u(t_0)) \text{.}
\end{align}
Furthermore, \eqref{xui8dddd9999} implies
\begin{align}\label{xui8ddddddd33d9999}
&\|\psi(\eta)\phi(\xi)Z^{a}u_{\xi}^{{T}}\widetilde{H}_a\|_{L^1_{t,x}(S_{t_0})} \nonumber \\
&\quad \leq \frac{1}{100}\|\langle\eta\rangle^{-\frac{1+\delta}{2}}\langle \xi\rangle^{1+\delta}Z^{a}u_{\xi}\|^2_{L^2_{t,x}(S_{t_0})} \nonumber \\
&\quad \qquad + C\sum_{|b|=2}\|\langle\xi\rangle^{-\frac{1+\delta}{2}}\langle \eta\rangle^{1+\delta}Z^{b}u_{\eta}\|^2_{L^2_{t,x}(S_{t_0})}+C\sum_{|b|\leq 1}\|\langle\xi\rangle^{-\frac{1+\delta}{2}}\langle \eta\rangle^{1+\delta}Z^{b}u_{\eta}\|^2_{L^2_{t,x}(S_{t_0})} \nonumber \\
&\quad \qquad + C\|\langle\eta\rangle^{-\frac{1+\delta}{2}}\langle \xi\rangle^{1+\delta}Z^{a}u_{\xi}\|_{L^2_{t,x}(S_{t_0})}\sum_{|b|\leq 1}\|\langle \eta\rangle^{1+\delta}Z^{b}u_{\eta}\|_{L^{\infty}_{t,x}(S_{t_0})} \nonumber \\
&\quad \qquad \qquad \cdot \sum_{|c|\leq 2}\|\langle\eta\rangle^{-\frac{1+\delta}{2}}\langle \xi\rangle^{1+\delta}Z^{c}u_{\xi}\|_{L^2_{t,x}(S_{t_0})} \nonumber \\
&\quad \qquad + C \|\langle\eta\rangle^{-\frac{1+\delta}{2}}\langle \xi\rangle^{1+\delta}Z^{a}u_{\xi}\|_{L^2_{t,x}(S_{t_0})}\sum_{|b|\leq 2}\|\langle \eta\rangle^{1+\delta}Z^{b}u_{\eta}\|_{L^{\infty}_{t}L^{2}_{x}(S_{t_0})} \nonumber \\
&\quad \qquad \qquad \cdot \sum_{|c|\leq 1}\|\langle\eta\rangle^{-\frac{1+\delta}{2}}\langle \xi\rangle^{1+\delta}Z^{c}u_{\xi}\|_{L^2_{t}L^{\infty}_{x}(S_{t_0})} \nonumber \\
&\quad \leq \frac{1}{100}\widehat{\mathcal {E}}_3(u(t_0))+C\overline{\mathcal {E}}_3(u(t_0))+C\widetilde{\mathcal {E}}_3(u(t_0)) + C \big[ {\mathcal {E}}_1(u(t_0))+{\mathcal {E}}_2(u(t_0)) \big] \nonumber \\
&\quad \qquad + C\sup_{0\leq t\leq t_0}E^{1/2}(u(t)) \cdot \mathcal{E}(u(t_0)) \text{.}
\end{align}

Thus, combining \eqref{itfolojjj8899jj888wjk99}--\eqref{xui8ddddddd33d9999}, \eqref{duix998709}, \eqref{xui9967eee444r88}, \eqref{xui9967eee44erfrrrrrr4r88}, and \eqref{Pthcom234888999}, we conclude that
\begin{align}\label{edddeejccfgg9955e}
&\widehat{{E}}_3(u(t_0))+\widehat{{{\mathcal{E}}}}_3(u(t_0)) \nonumber \\
&\quad \leq C\varepsilon^2+C\overline{E}_3(u(t_0))+C\overline{\mathcal {E}}_3(u(t_0))+C\widetilde{E}_3(u(t_0))+C\widetilde{\mathcal {E}}_3(u(t_0)) \nonumber \\
&\quad \qquad + C \big[ {\mathcal {E}}_1(u(t_0))+{\mathcal{E}}_2(u(t_0)) \big] + C E^{3/2}(u(t_0)) + C \sup_{0 \leq t \leq t_0} E^{1/2}(u(t)) \cdot \mathcal{E}(u(t_0)) \nonumber \\
&\quad \leq C\varepsilon^2+CA^3\varepsilon^3 \text{.}
\end{align}

\subsection{Conclusion of the Proof}

Combining the estimates \eqref{duix998709}, \eqref{Pthcom234888999}, \eqref{xui9967eee444r88}, \eqref{xui9967eee44erfrrrrrr4r88}, and \eqref{edddeejccfgg9955e} yields
\begin{align}
\sup_{0 \leq t \leq T} \big[ E(u(t))+\mathcal {E}(u(t)) \big] &\leq C_1 \varepsilon^2 + C_2 A^3\varepsilon^3 \text{,}
\end{align}
for some constants $C_1$ and $C_2$.
Taking $A^2 = 2 \max \{C_0, C_1\}$ and $\varepsilon_0$ sufficiently small such that $2 C_2 A \varepsilon_0 \leq 1$ results in the bound \eqref{rgty6788889} and completes the bootstrap argument.

To complete the proof of Theorem \ref{thn88786hhh9900}, it remains only to show that our solution $u$ can be extended beyond our given time $T$.
Then, the above bootstrap estimates (which are independent of $T$) can be iterated indefinitely to prove global existence and smallness.

This extension can be established, for instance, using the classical local well-posedness theory from \cite{hkm:ql_lwp}.
We briefly give an informal sketch of this argument below.

\subsubsection{Sketch of Local Well-Posedness}

First, we write our system \eqref{vsystem99999} in terms of Cartesian derivatives,
\begin{equation}
\label{vsystem100000} a_{00} ( u_\xi, u_\eta ) \, u_{tt} = a_{11} ( u_\xi, u_\eta ) \, u_{xx} + ( a_{01} + a_{10} ) \, u_{tx} + \text{lower-order terms.}
\end{equation}
Then, from \cite[Section 3.1]{hkm:ql_lwp}, the key assumptions needed for local well-posedness are
\begin{equation}
\label{hkm_ass} y^T \, a_{00} ( u_\xi, u_\eta ) \, y \geq \mathcal{C} | y |^2 \text{,} \qquad y^T \, a_{11} ( u_\xi, u_\eta ) \, y \geq \mathcal{C} | y |^2 \text{,}
\end{equation}
for all $y \in \mathbb{R}^n$ and for some constant $\mathcal{C} > 0$.
Since $u_\xi$ and $u_\eta$ are guaranteed to be as small as we need in our setting, then by continuity, it suffices to show, for some $\mathcal{C} > 0$,
\begin{equation}
\label{hkm_real_ass} y^T \, a_{00} ( 0, 0 ) \, y \geq \mathcal{C} | y |^2 \text{,} \qquad y^T \, a_{11} ( 0, 0 ) \, y \geq \mathcal{C} | y |^2 \text{.}
\end{equation}

Expressing in terms of $\xi$ and $\eta$, the two conditions in \eqref{hkm_real_ass} expand to
\begin{align}
\label{hkm_real_ass_1} y^T [ 1 - A_1 ( f' ( \xi ), 0 ) - A_2 ( f' ( \xi ), 0 ) ] y &\geq \mathcal{C} | y |^2 \text{,} \\
\label{hkm_real_ass_2} y^T [ 1 - A_1 ( f' ( \xi ), 0 ) + A_2 ( f' ( \xi ), 0 ) ] y &\geq \mathcal{C} | y |^2 \text{,}
\end{align}
the first of which is our assumption \eqref{hyper1}.
If \eqref{hkm_real_ass_2} also holds, then the well-posedness result of \cite[Theorem III]{hkm:ql_lwp} holds here, and our proof of Theorem \ref{thn88786hhh9900} is complete.

On the other hand, if \eqref{hkm_real_ass_2} fails to hold, then we define the change of variables
\begin{equation}
\bar{t} = ( 1 + c ) t \text{,} \qquad \bar{x} = x + c t \text{,}
\end{equation}
with $0 < c < 1$, from which we obtain a transformed system
\begin{equation}
\label{vsystem100001} a_{00} ( u_\xi, u_\eta ) \, u_{ \bar{t} \bar{t} } = a_{11} ( u_\xi, u_\eta ) \, u_{ \bar{x} \bar{x} } + ( a_{01} + a_{10} ) \, u_{ \bar{t} \bar{x} } + \text{lower-order terms,}
\end{equation}
with a different set of coefficients $a_{00}$, $a_{01}$, $a_{10}$, $a_{11}$.
A direct computation shows that the positivity conditions \eqref{hkm_real_ass}, in the new $( \bar{t}, \bar{x} )$-coordinates, are now given by
\begin{align}
\label{hkm_real_ass_3} ( 1 + c )^2 y^T [ 1 - A_1 ( f' ( \xi ), 0 ) - A_2 ( f' ( \xi ), 0 ) ] y &\geq \mathcal{C} | y |^2 \text{,} \\
\label{hkm_real_ass_4} ( 1 - c ) y^T [ ( 1 + c ) ( 1 - A_1 ( f' ( \xi ), 0 ) ) + ( 1 - c ) A_2 ( f' ( \xi ), 0 ) ] y &\geq \mathcal{C} | y |^2 \text{.}
\end{align}
In particular, by \eqref{hyper1} and \eqref{hyper3}, both \eqref{hkm_real_ass_3} and \eqref{hkm_real_ass_4} hold as long as $c$ is sufficiently close to $1$.
Thus, applying \cite[Theorem III]{hkm:ql_lwp} completes the proof of Theorem \ref{thn88786hhh9900} in general.

\begin{rem}
In particular, notice that $\bar{t}$ and $t$ have the same level sets, and $\partial_{ \bar{\xi} }$ points in the same direction as $\partial_\xi$, so that one still captures the traveling wave.
The rough idea here is to ensure that the transformed $\partial_{ \bar{t} }$ is now pointing in a ``timelike" direction.
\end{rem}

\section*{Acknowledgments}

This work was carried out while the first author was visiting the School of Mathematical Sciences at Queen Mary University of London from January 2020.
He would like to thank the support of the China Scholarship Council, as well as the school's hospitality.
The first author is supported by the National Natural Science Foundation of China No.~11801068 and the Fundamental Research Funds for the Central Universities.
The second author is supported, for a portion of this work, by EPSRC grant EP/R011982/1.

%\bibliographystyle{wave}
%\bibliography{Quasilinear123}

\end{document}